\documentclass[a4paper,11pt,leqno]{amsart}
\usepackage{amsmath}
\usepackage{amsthm}
\usepackage{amsfonts}
\usepackage{amssymb}
\usepackage{texdraw}
\usepackage{mathbbol}
\usepackage{txfonts}
\usepackage{mathrsfs}

\newtheorem{theorem}{Theorem}[section]{\bf}{\it}
\newtheorem{lemma}[theorem]{Lemma}{\bf}{\it}
\newtheorem{proposition}[theorem]{Proposition}{\bf}{\it}
\newtheorem{corollary}[theorem]{Corollary}{\bf}{\it}
\newtheorem{example}[theorem]{Example}{\bf}{\it}
{\bf}{\it} 
{\bf}{\it}

\newtheorem{defn}[theorem]{Definition}
\theoremstyle{remark}

\newtheorem{remark}[theorem]{Remark}

\numberwithin{equation}{section}

\newcommand{\R}{\mathbb R}

\newcommand{\N}{\mathbb N}

\newcommand{\capa}{\operatorname{cap}}

\newcommand{\spt}{\operatorname{spt}}
\newcommand{\set}{\operatorname{set}}

\newcommand{\dist}{{\operatorname{dist}\,}}

\newcommand{\fillvol}{{\operatorname{Fillvol}}}
\newcommand{\length}{\ell}

\newdimen\vintkern\vintkern11pt
\def\vint{-\kern-\vintkern\int}

\newcommand{\norm}[1]{\lVert #1 \rVert}
\newcommand{\md}{\operatorname{md}}

\newcommand{\haus}{\mathcal{H}}

\newcommand{\Z}{\mathbb{Z}}

\newcommand{\ptLip}{\operatorname{Lip}} 
\newcommand{\ptlowlip}{\operatorname{lip}}
\newcommand{\Lipconst}{\operatorname{Lip}} 
\newcommand{\Lipspace}{\operatorname{Lip}} 

\newcommand{\AKc}{{\mathbf M}}    
\newcommand{\AKnc}{{\mathbf N}}   
\newcommand{\AKirc}{{\mathcal I}} 
\newcommand{\AKic}{{\mathbf I}}   

\newcommand{\mass}{{\mathbf M}} 
\newcommand{\nmass}{{\mathbf N}} 
\newcommand{\bdry}{{\partial}} 
\newcommand{\flatnormSCR}{{\mathscr F}} 
\newcommand{\flatnormFAT}{{\mathbf F}} 

\newcommand{\rstr}{\:\mbox{\rule{0.1ex}{1.2ex}\rule{1.1ex}{0.1ex}}\:} 

\newcommand{\lm}{{\mathscr L}} 

\begin{document}

\title{An upper gradient approach to weakly differentiable cochains}
\date{}
\author{Kai Rajala \and Stefan Wenger}

\thanks{K.R. was supported by the Academy of Finland. Parts of this research were carried out when K.R. was visiting University of Illinois at Chicago and University of Michigan. He thanks the departments for their hospitality. S.W. was partially supported by NSF grants DMS--1056263 and DMS--0956374; parts of this research were carried out while S.W. was an Assistant Professor at the University of Illinois at Chicago. He would like to thank the department for the excellent research environment. He would moreover like thank the University of Jyv\"askyl\"a for its hospitality during a visit when parts of this research were carried out.}  
\subjclass[2010]{49Q15, 46E35, 53C65, 49J52, 30L99}

\begin{abstract}
The aim of the present paper is to define a notion of weakly differentiable cochain in the generality of metric measure spaces and to study basic properties of such cochains. Our cochains are (sub-)linear functionals on a subspace of chains, and a suitable notion of chains in metric spaces is given by Ambrosio-Kirchheim's theory of metric currents. The notion of weak differentiability we introduce is in analogy with Heinonen-Koskela's concept of upper gradients of functions. 
In one of the main results of our paper, we prove continuity estimates for cochains with $p$-integrable upper gradient in $n$-dimensional Lie groups endowed with a left-invariant Finsler metric. Our result generalizes the well-known Morrey-Sobolev inequality for Sobolev functions. Finally, we prove several results relating capacity and modulus to Hausdorff dimension.

\end{abstract}

\maketitle


\section{Introduction}

\subsection{Background}
One of the main principles in the theory of Sobolev functions in euclidean spaces is that good integrability properties of the weak differential of a function implies good behavior for the function itself. For instance, Sobolev inequalities bound the values of the function in terms of the integral of the gradient. In particular, the Morrey-Sobolev inequality shows that a weakly differentiable function $u \in L^1_{\operatorname{loc}}(\R^n)$ with $|\nabla u| \in L^p(\R^n)$ has a H\"older continuous representative when $p>n$, 
\begin{equation}
\label{morrey}
|u(x)-u(y)| \leq C(n,p)|x-y|^{1-n/p}\norm{\nabla u}_p. 
\end{equation}
An appealing question is whether continuity results like this also hold in the case of differential forms. Namely, given an $m$-form $\omega$, we can view it as a linear functional defined on a class of $m$-dimensional chains (smooth submanifolds, polyhedral chains, currents, etc.). We can now ask for conditions on the coefficients of $\omega$ which guarantee continuity of this functional with respect to a suitable metric. An important condition like this is given by Whitney's theory of flat forms. By definition, these are the $m$-forms $\omega$ whose coefficients, as well as the coefficients of the distributional exterior derivative, are essentially bounded. By Wolfe's theorem \cite[p. (viii)]{Whitney}, \cite[Theorem 5.5]{Heinonen}, the  space of flat forms is isomorphic to the space of flat cochains. These are bounded linear functionals on the space of flat chains, the completion of polyhedral $m$-chains with respect to the flat norm 
$$
 \flatnormFAT(T):= \inf\{ \mass(R)+\mass(V):\, T=R+\bdry V  \}. 
$$
It follows that integration of a flat form $\omega$ over any flat chain is well-defined although the coefficients of $\omega$ are initially only defined pointwise almost everywhere. Moreover, it follows that flat forms, when viewed as cochains, are Lipschitz continuous with respect to the flat norm. We note that the theory of flat forms has recently been extended to Banach spaces in \cite{Snipes}. 

Recently, a theory of Sobolev spaces in metric measure spaces $(X,d,\mu)$ has been developed based on \emph{upper gradients}, see \cite{HeiKos}, \cite{Shan}, \cite{HKST}, and the forthcoming monograph \cite{HKSTbook}.
By definition, a non-negative Borel function $\rho$ is an upper gradient of a function $u:X \to \overline{\R}$ if 
$$
|u(y)-u(x)| \leq \int_{\gamma} \rho \, ds 
$$
for every $x$ and $y \in X$ and every rectifiable path $\gamma$ in $X$ with endpoints $x$ and $y$. We say that $u \in L^p(X,\mu)$ belongs to the \emph{Newtonian} (Sobolev) space $N^{1,p}(X,\mu)$ if $u$ has an upper gradient $\rho \in L^p(X,\mu)$. This approach works in general spaces, even when directional derivatives cannot be defined. It also gives a useful viewpoint in smooth spaces, where the Newtonian spaces coincide with classical Sobolev spaces. The theory includes several generalizations of the Sobolev inequalities, as well as the continuity estimate \eqref{morrey}, under mild assumptions on the underlying metric measure space, cf. \cite{Hajlasz-Koskela} and the references therein. 

The aim of the present paper is to generalize the results discussed above. Namely, we address the following problems: 
\begin{itemize}
\item[(i)] give a proper notion for weakly differentiable $m$-forms in metric measure spaces using the upper gradient approach, and prove useful properties for them, in particular 
\item[(ii)] find $L^p$-conditions which imply continuity.  
\end{itemize}
Problem (ii) is interesting already in euclidean spaces. Our main results give continuity estimates with respect to the flat norm and the so-called filling volume in euclidean spaces and Lie groups; we will discuss these results shortly. 

We now turn to Problem (i). As discussed above, differential forms induce linear functionals defined on $m$-dimensional chains. Such functionals can be defined without assuming any structure from the underlying space. Therefore, we would like to define cochains $\omega: \mathcal{C} \to \overline{\R}$, where $\mathcal{C}$ is a suitable family of $m$-dimensional chains, and try to develop their properties. A question that immediately comes up in this approach is how to find a suitable notion of $m$-chains. Such a notion in the generality of complete metric spaces is provided by Ambrosio-Kirchheim's theory of metric currents developed in \cite{Ambrosio-Kirchheim-currents} which we next discuss. 

\subsection{Metric currents}

We recall that a Federer-Fleming $m$-current in $\R^n$ is a continuous linear functional on the space of compactly supported smooth differential $m$-forms. 
In the generality of a complete metric space $X$ a suitable substitute for $m$-forms is given by $(m+1)$-tuples $(f,\pi_1,\dots, \pi_m)$ of Lipschitz functions on $X$ with $f$  bounded. A metric $m$-current in the sense of Ambrosio-Kirchheim \cite{Ambrosio-Kirchheim-currents} is then a multi-linear functional on such tuples which satisfies a continuity, locality and finite mass property. We refer to Section~\ref{sec:metriccurrents} below for definitions. The space of metric $m$-currents in $X$ is denoted by $\AKc_m(X)$. Metric currents have finite mass by definition and the mass as a measure of $T\in\AKc_m(X)$ is denoted by $\|T\|$; furthermore $\mass(T):= \|T\|(X)$. The boundary of an element $T\in\AKc_m(X)$ with $m\geq 1$ is denoted by $\bdry T$. A metric $m$-current $T$ whose boundary $\bdry T$ has finite mass is called normal current; and the space of such $T$ is denoted by $\AKnc_m(X)$. 
One of the guiding principles is that in Euclidean space a tuple $(f,\pi_1, \dots, \pi_m)$ with $f$ and $\pi_i$ smooth should correspond to the differential form $fd\pi_1\wedge\dots\wedge d\pi_m$ and tuples $(f,\pi_1,\dots,\pi_m)$ may thus be regarded as generalized differential forms. 
An important subclass of normal $m$-currents is given by the additive subgroup $\AKic_m(X)\subset\AKnc_m(X)$ of integral $m$-currents. These are normal currents which roughly correspond to (integration of the generalized forms over) countably $\haus^m$-rectifiable sets with orientation and integer multiplicities. In particular, $0$-dimensional integral currents correspond to points with integer weights. Moreover, Lipschitz curves give rise to $1$-dimensional integral currents; and in fact a weak converse of this is true as well, see Lemma~\ref{lem:weak-structure-1-currents} and \cite[Lemma 4.4]{Ambrosio-Wenger}. 

\subsection{Weakly differentiable cochains}
\label{section:introforms}
We now turn to the main object of study of the present paper, namely $m$-cochains. For this let $\mathcal{C}_m$ be an additive subgroup of $\AKc_m(X)$. We call cochain on $\mathcal{C}_m$ a function $\omega:\mathcal{C}_m\to\overline{\R}$ which satisfies $\omega(0)=0$ and which is sublinear in the sense that $$|\omega(T)|\leq |\omega(T+S)| + |\omega(S)|$$ for all $T,S\in\mathcal{C}_m$. If $\omega$ furthermore satisfies $\omega(T+S) = \omega(T) + \omega(S)$ for all $T, S$ for which each term is finite, then $\omega$ will be called a linear cochain. Clearly, every generalized $m$-form $(f,\pi_1,\dots,\pi_m)$ gives rise to a linear cochain on $\AKc_m(X)$ by $\omega(T)=T(f,\pi_1,\dots,\pi_m)$. Moreover, every function $u:X\to\R$, even if not Lipschitz, gives rise to a cochain on $\AKic_0(X)$. More examples will be given later. 

We can define the notion of upper gradient of a cochain in analogy with the definition of upper gradient of a function. For this, let $\mathcal{C}_{m+1}\subset\AKc_{m+1}(X)$ and let $\omega$ be a cochain on $\mathcal{C}_m$. We call a Borel function $g:X\to[0,\infty]$ an upper gradient of $\omega$ with respect to $\mathcal{C}_{m+1}$ if 
\begin{equation*}
 |\omega(T)|\leq \int_Xg d\|S\|
\end{equation*}
for all $T\in\mathcal{C}_m$ and $S\in\mathcal{C}_{m+1}$ satisfying $\bdry S = T$. This definition of upper gradient may be viewed as a generalization of the notion of upper gradient of a function. Indeed, we will show in Proposition~\ref{prop:upper-grad-lipfct-0-cochain} that a Borel function $g$ is an upper gradient of a function $u:X\to\overline{\R}$ if and only if $g$ is an upper gradient of the cochain on $\AKic_0(X)$ induced by $u$. We will moreover show that if $m\geq 0$ and if $(f,\pi_1,\dots, \pi_m)$ is a generalized differential form then an upper gradient of the $m$-cochain on $\AKic_m(X)$ induced by $(f,\pi_1,\dots,\pi_m)$ is given by the product 
\begin{equation}
\label{ugex}
g(x) = \ptlowlip f(x)\prod_{i=1}^m\ptlowlip \pi_i(x) 
\end{equation} 
of pointwise lower Lipschitz constants, see Proposition~\ref{prop:upper-grad-norm-lowlip}. This is a generalization for cochains of the fact, proved by Cheeger in \cite{Cheeger-diff}, that if $f$ is a Lipschitz function on $X$ then the pointwise lower Lipschitz constant $\ptlowlip f(\cdot)$ is an upper gradient of $f$. In Proposition~\ref{prop:upper-grad-norm-Lipform-cochain} we establish an analogous result for cochains on $\AKc_m(X)$.

Similarly, we can define an upper norm of a cochain $\omega$ on $\mathcal{C}_m$. We call a Borel function $h:X \to [0,\infty]$ an upper norm of $\omega$ if 
$$
|\omega(T)| \leq \int_{X} h \, d\norm{T} 
$$
for all $T \in \mathcal{C}_m$. For example, the function 
$$
h(x) = |f(x)|\prod_{i=1}^m\ptlowlip \pi_i(x) 
$$
is an upper norm of the cochain described before \eqref{ugex}. We will give more examples of upper norms and upper gradients later. 

We have now given the necessary definitions that allow us to talk about weakly differentiable cochains in metric measure spaces; they are the cochains with integrable upper gradients and/or integrable upper norms. Our purpose is to show that analytic properties for the cochains can be deduced using the properties of their upper gradients and upper norms.

\subsection{Continuity of cochains in Lie groups}

One of the main goals of this paper is to establish continuity estimates with respect to the filling volume for cochains with $p$-integrable upper gradient. For this purpose we denote by $\AKic_m^0(X)$ the subset of elements $T\in\AKic_m(X)$ with $\bdry T=0$. We furthermore recall that the filling volume of an element $T\in\AKic_m^0(X)$ is defined by 
$$
 \fillvol(T):= \inf\{\mass(S): S\in\AKic_{m+1}(X), \bdry S = T\}.
$$
In a slightly simplified setting, one of our main results can be stated as follows.

\begin{theorem} \label{thm:simplified}
Let $G$ be a Lie group, endowed with a left-invariant Riemannian metric, and let $0\leq \alpha \leq m$ and $A\geq 1$. Let $\omega$ be a cochain on $\AKic_m^0(G)$. If $\omega$ has an upper gradient $g$ in $L^p(G)$ for some $p>n-\alpha$ then 
\begin{equation}\label{eq:thm-simpli-Hoelder-eq}
 |\omega(T)| \leq C\, \fillvol(T)^{1- \frac{n}{p+\alpha}} \norm{g}_p
\end{equation}
for every $T\in\AKic_m^0(G)$ which satisfies $\fillvol(T)\leq 1$ and 
\begin{equation}\label{eq:simpli-growth-T}
\norm{T}(B(x,r))\leq Ar^{\alpha} \quad \text{for all }x \in \R^n \text{ and } r>0. 
\end{equation}
Here, $C$ depends only on $\mass(T)$, $n$, $p$, and $\alpha$, $A$, and $G$. 
\end{theorem}

The precise value of $C$ is given in Theorem~\ref{hoeldercont2}. The requirement that $\fillvol(T)\leq 1$ can be dropped if $G=\R^n$ is Euclidean space. We note that if $\omega$ is a linear cochain and if $T_1,T_2\in\AKic_m^0(G)$ satisfy \eqref{eq:simpli-growth-T} and $$d_{\mathrm{F}}(T_1,T_2):= \fillvol(T_1-T_2)\leq 1$$ then \eqref{eq:thm-simpli-Hoelder-eq} can be written in the more suggestive form 
$$
|\omega(T_1) - \omega(T_2)| \leq C d_{\mathrm{F}}(T_1,T_2)^{1- \frac{n}{p+\alpha}} \norm{g}_p,
$$
and thus $\omega$ is locally H\"older continuous with respect to the metric $d_{\mathrm{F}}$. It should be noted that Theorem~\ref{thm:simplified} fails for $p=n-\alpha$, see Example~\ref{sharpthm}. We do not know, however, whether the H\"older exponent $1-\frac{n}{p+\alpha}$ can be improved and to what extent the growth bound for $T$ is necessary. In Theorem~\ref{thm:simplified} we will assume a growth condition which is somewhat weaker than the one in \eqref{eq:simpli-growth-T}. It is easy to see that Theorem~\ref{thm:simplified} implies the local Morrey-Sobolev inequality for functions in $W^{1,p}(G)$ with $p>n$, see Corollary~\ref{cor:Morrey-Sobolev}.

In Section~\ref{part21} we will also establish a theorem for currents, possibly with boundary, which is similar to Theorem~\ref{thm:simplified} and which gives H\"older continuity with respect to the flat norm rather than the filling volume distance, see Theorem~\ref{hoeldercont1}. This is natural in view of the Lipschitz continuity of flat forms with respect to the flat norm mentioned above. As will be shown, our result actually implies that every flat form with compact support in $\R^n$ gives rise to a cochain which is Lipschitz continuous with respect to the flat norm, and we can thus recover a part of Wolfe's theorem mentioned above. See the paragraph following Theorem~\ref{hoeldercont1} for details.

Similar, but less general results than ours have been previously obtained in \cite{Goldshteinetal}. There it is assumed that $\omega$ belongs to the Sobolev space $W^{q,p}_d(\R^n,\bigwedge^m)$ of $m$-forms in $\R^n$ whose coefficients are $q$-integrable and the coefficients of the weak exterior derivative are $p$-integrable with $p>n-m$ and $q>n-m+1$. It is then proved that, given an oriented $m$-ball $B$ in $\R^n$, the integral of $\omega$ over $B$ is bounded by the corresponding $p$- and $q$-integrals over a suitable domain, the radius of $B$, and the size of the domain. 

\subsection{Sobolev forms and exceptional sets} 
The weakly differentiable cochains defined in Section \ref{section:introforms} are closely connected to the Sobolev spaces $W^{q,p}_d(\R^n,\bigwedge^m)$. See \cite{IwaLuto} for a good reference on Sobolev spaces of differential forms. In Section \ref{exception} we show that, when $1 < p,q<\infty$, every $\omega \in W^{q,p}_d(\R^n,\bigwedge^m)$ induces a linear, weakly differentiable cochain $\tilde{\omega}$, and in fact the norms $|\omega|$ and $|d\omega|$ are an (weak) upper norm and upper gradient of $\tilde{\omega}$, respectively (up to a constant depending on the choice of norms for $\omega$ and $d\omega$). 
We believe that, conversely, linear weakly differentiable cochains probably come from such forms, but we do not pursue this direction in this paper. A result in this spirit has been established in \cite{Goldshteinetal2}. There a version of Wolfe's theorem is proved, showing that there is a one-to-one correspondence between the $W^{q,p}_d(\R^n,\bigwedge^m)$-forms and cochains defined on polyhedral chains who together with their exterior derivatives satisfy certain boundedness conditions with respect to the so-called $q$-mass. 

In the theory of Sobolev functions, capacities are typically used to measure the size of exceptional sets. For instance, the Morrey-Sobolev inequality \eqref{morrey} corresponds to the fact that the $p$-capacity of a single point is positive when $p>n$, and there are weak forms of \eqref{morrey} for smaller $p$ which hold outside a set of zero $p$-capacity. In the theory based on upper gradients, the modulus of path families is an important concept that can be applied in connection with exceptional sets. 

Modulus methods can be extended to much beyond the setting of path families, as already observed by Fuglede \cite{Fuglede-modulus}. In our current setting, the definition is the following. Let $X$ be a complete metric space equipped with a Borel measure $\mu$. Moreover, let $\Gamma \subset \AKc_m(X)$ be a family of currents, and $1 \leq p < \infty$. The $p$-modulus $M_p(\Gamma)$ is the infimum $\int_X f^p \, d\mu$, taken over all non-negative Borel functions $f$ in $X$, such that $\int_X f \, d\norm{T} \geq 1$ for all $T \in \Gamma$. Modulus in the setting of currents implicitly appears in \cite{Williams}, where nonexistence and other results are proved for currents in Carnot groups. 

Similarly, let $\Lambda \subset \AKc_m(X)$ be a family of currents without boundary, and let $\mathcal{C}' \subset \AKc_{m+1}(X)$. Then we can define the $p$-capacity $\operatorname{cap}_p(\Lambda, \mathcal{C}')$ as $M_p(\Gamma)$, where 
$$
\Gamma = \{S \in \mathcal{C}': \bdry S =T \text{ for some } T \in \Lambda \}. 
$$  
In Theorem~\ref{thm:hausdim-capacity-currents}, we relate Hausdorff measure and capacity. Namely, we show that a family of integral currents, all of whose supports lie on a compact set $A \subset X$ with $\haus^{Q-p}(A)< \infty$, has zero $p$-capacity if the underlying measure $\mu$ satisfies $\mu(B(x,r))\leq C r^Q$ for all balls in $X$. In Section \ref{section:rel-haus-cap}, we consider capacity in the setting of Lie groups. We show that if $T$ is a current as in Theorem~\ref{thm:simplified}, then the $p$-capacity of $\{T\}$ is positive if $p>n-\alpha$. This is not surprising in view of Theorem \ref{thm:simplified}. We also give an example to show the above can fail when $p<n-\alpha$; it is not completely clear to us what happens when $p=n-\alpha$. Our results are related to those by Fuglede \cite{Fuglede-modulus}, who gave necessary and sufficient conditions under which the modulus of the family of all Lipschitz surfaces in $\R^n$ intersecting a given set has zero modulus. 


\subsection{Organization of the paper}

This paper is structured as follows. In Section \ref{sec:metriccurrents} we recall the definition of metric currents and some of the basic properties needed later on. In Section \ref{part1} we discuss cochains in general metric spaces, and give some basic examples. First, in \ref{part11} we define cochains, upper gradients and upper norms. We also discuss the modulus and capacity in our context, and the spaces of cochains with integrable upper norms and upper gradients. In \ref{exception} we define Sobolev spaces of linear cochains and show that Euclidean differential forms with integrable distributional exterior derivatives are examples of Sobolev cochains. In \ref{section:estimtes-upper-norm-grad} we give basic examples of upper norms and upper gradients, and compare them to upper gradients of functions in the zero-dimensional case. In \ref{section:rel-haus-cap}, we prove upper bounds for the sizes of exceptional sets. 

In Section \ref{part2} we prove general versions of the continuity estimate, Theorem \ref{thm:simplified}, in Lie groups. To this end, in \ref{part22} we first establish integral estimates corresponding to general measures on Lie groups, and define a ``controlled family of curves'', a condition that allows us to deform currents in a controlled way. In \ref{section:tech-est-cochains} we estimate cochains with integrable upper norms and upper gradients by using translations and minimal fillings, and use the estimates to prove the continuity statements. Finally, in \ref{section:lie-haus-cap} we prove lower bounds for modulus and capacity. 



\section{Preliminaries}

\subsection{Notation}\label{sec:notation}
Let $(X,d)$ be a metric space. Given $x\in X$ and $r>0$ we denote by $B(x,r)$ the closed ball $B(x,r):= \{y\in X: d(x,y)\leq r\}$. Given a set $A\subset X$ and $x\in X$ we define $\dist(x,A):= \inf\{r\geq 0 : \exists a\in A \text{ with } d(x,a)\leq r\}$. For $\varepsilon>0$ we then denote $N(A,\varepsilon):= \{x\in X: \dist(x, A)\leq \varepsilon\}.$
Given $\alpha\geq 0$ and $A\subset X$ we denote by $\haus^\alpha(A)$ the $\alpha$-Hausdorff measure of $A$.
We denote by $\Lipspace(X)$ and $\Lipspace_b(X)$ the spaces of real-valued Lipschitz functions and bounded Lipschitz functions on $X$, respectively. The Lipschitz constant of a Lipschitz function $f$ will be denoted by $\Lipconst(f)$. The length of a continuous curve $c:[a,b]\to X$ is denoted by $\length(c)$. If $c$ is a Lipschitz curve then the metric derivative of $c$ is denoted $$|\dot{c}|(t) = \lim_{r\to 0} \frac{1}{r}d(c(t+r), c(t)),$$ whenever the limit exists. It is proved in \cite{Kirchheim-rectifiable} that $|\dot{c}|(t)$ exists for almost every $t\in[a,b]$.

\subsection{Currents in metric spaces}
\label{sec:metriccurrents}

In this section we recall the basic definitions from the theory of metric currents developed in \cite{Ambrosio-Kirchheim-currents} which we will need in the sequel. Apart from some simple lemmas, the present section does not contain any new results. 
We mention here that recently two variants of Ambrosio-Kirchheim's theory \cite{Ambrosio-Kirchheim-currents}  were developed in \cite{Lang-currents} and \cite{Lang-Wenger-pted-cptness}. We will not however use these variants.

Let $(X,d)$ be a complete metric space. 
\begin{defn}\label{def:current}
Let $m\geq 0$. An $m$-dimensional metric current  $T$ on $X$ is a multi-linear functional $T:\Lipspace_b(X)\times\Lipspace^m(X)\to\R$ satisfying the following
properties:
\begin{enumerate}
 \item If $\pi^j_i\to \pi_i$ pointwise as $j\to\infty$ and if $\sup_{i,j}\Lipconst(\pi^j_i)<\infty$ then
       \begin{equation*}
         T(f,\pi^j_1,\dots,\pi^j_m) \longrightarrow T(f,\pi_1,\dots,\pi_m).
       \end{equation*}
 \item If $\{x\in X:f(x)\not=0\}$ is contained in the union $\bigcup_{i=1}^mB_i$ of Borel sets $B_i$ and if $\pi_i$ is constant 
       on $B_i$ for $i=1, \dots, m$ then
       \begin{equation*}
         T(f,\pi_1,\dots,\pi_m)=0.
       \end{equation*}
 \item There exists a finite Borel measure $\mu$ on $X$ such that
       \begin{equation}\label{equation:mass-def}
        |T(f,\pi_1,\dots,\pi_m)|\leq \prod_{i=1}^m\Lipconst(\pi_i)\int_X|f|d\mu
       \end{equation}
       for all $(f,\pi_1,\dots,\pi_m)\in\Lipspace_b(X)\times\Lipspace^m(X)$.
\end{enumerate}
\end{defn}
In what follows, $m$-dimensional metric currents will also be called metric $m$-currents for short. The space of $m$-dimensional metric currents on $X$ is denoted by $\AKc_m(X)$ and the minimal Borel measure $\mu$
satisfying \eqref{equation:mass-def} is called mass of $T$ and denoted by $\|T\|$. We also call mass of $T$ the number $\|T\|(X)$ 
which we denote by $\mass(T)$.
The support of $T$ is  the closed set $$\spt T = \{x\in X:  \text{ $\|T\|(B(x,r))>0$ for all $r>0$}\}.$$ 

In the following we will often abbreviate $\pi=(\pi_1, \dots, \pi_m)$ and write $T(f,\pi)$ instead of $T(f,\pi_1,\dots, \pi_m)$. An important and basic example of a metric $m$-current on $\R^m$ is given by $$\Lbrack\theta\Rbrack(f, \pi):= \int_{\R^m}\theta f\det\left(\nabla\pi\right)\,d\haus^m$$ for an arbitrary function $\theta\in L^1(\R^m)$.

Let $0\leq k\leq m$. Given a bounded Borel function $g$ on $X$ and $\tau = (\tau_1,\dots, \tau_k)\in\Lipspace^k(X)$, the restriction $T\rstr (g,\tau)$ of an element $T\in\AKc_m(X)$ is defined by
\begin{equation*}
 (T\rstr(g,\tau))(f,\pi):= T(fg, \tau, \pi)
\end{equation*}
for all $(f,\pi) \in \Lipspace_b(X)\times\Lipspace^{m-k}(X)$. This expression is well-defined since $T$ can be extended to a functional on tuples for which the first argument lies in 
$L^\infty(X,\|T\|)$; in fact, we have $T\rstr (g,\tau)\in\AKc_{m-k}(X)$ by \cite[Theorem 3.5]{Ambrosio-Kirchheim-currents}. For a Borel set $A\subset X$ we abbreviate $T\rstr A:= T\rstr 1_A$, where $1_A$ is the indicator function,
\begin{equation*}
  (T\rstr A)(f,\pi):= T(f1_A,\pi).
\end{equation*}

If $m\geq 1$ and $T\in\AKc_m(X)$ then the boundary of $T$ is the functional
\begin{equation*}
 \bdry T(f,\pi_1,\dots,\pi_{m-1}):= T(1,f,\pi_1,\dots,\pi_{m-1});
\end{equation*}
it satisfies conditions (i) and (ii) in Defintion~\ref{def:current}. If it moreover satisfies (iii) in Definition~\ref{def:current} then $T$ is called a normal current.
By convention, elements of $\AKc_0(X)$ are also called normal currents. The space of normal metric $m$-currents on $X$ is denoted by $\AKnc_m(X)$. If $m\geq 2$ and $T\in\AKc_m(X)$ then we have $\bdry \bdry T = 0$ by property (ii) of Definition~\ref{def:current}. The following convention will be useful in Section~\ref{part2}. If $T\in\AKc_0(X)$ then we define $\mass(\bdry T)=0$ as a number and we define $\|\bdry T\|=0$ as a measure on $X$.

The push-forward of $T\in\AKc_m(X)$ 
under a Lipschitz map $\varphi$ from $X$ to another complete metric space $Y$ is given by
\begin{equation*}
 \varphi_\# T(g,\tau):= T(g\circ\varphi, \tau\circ\varphi)
\end{equation*}
for $(g,\tau)\in\Lipspace_b(Y)\times\Lipspace^m(Y)$. This defines a metric $m$-current on $Y$ and it follows directly from the definitions that $\bdry(\varphi_{\#}T) = \varphi_{\#}(\bdry T)$.

An element $T\in\AKc_0(X)$ is called integer rectifiable if there exist finitely many points $x_1,\dots,x_n\in X$ and $\theta_1,\dots,\theta_n\in\Z\backslash\{0\}$ such
that
\begin{equation}\label{eq:int-curr-0-dim}
 T(f)=\sum_{i=1}^n\theta_if(x_i)
\end{equation}
for every bounded Lipschitz function $f$.
 A current $T\in\AKc_m(X)$ with $m\geq 1$ is called integer rectifiable if the following properties hold:
 \begin{enumerate}
  \item $\|T\|$ is concentrated on a countably $\haus^m$-rectifiable set and vanishes on all $\haus^m$-negligible Borel sets;
  \item for any Lipschitz map $\varphi:X\to\R^m$ and any open set $U\subset X$ there exists $\theta\in L^1(\R^m,\Z)$ such that 
    $\varphi_\#(T\rstr U)=\Lbrack\theta\Rbrack$.
 \end{enumerate}
The space of integer rectifiable $m$-currents in $X$ is denoted by $\AKirc_m(X)$.
Integer rectifiable normal currents are called integral currents. The corresponding space is denoted by $\AKic_m(X)$. 
We introduce the notation $$\AKnc^0_m(X):= \{T\in\AKnc_m(X): \bdry T = 0\}$$ and $$\AKic^0_m(X):= \{T\in\AKic_m(X): \bdry T = 0\}.$$ Here, the condition $\bdry T=0$ should be replaced by the condition $T(1)=0$ in the case $m=0$.
More generally, if $\mathcal{C}\subset\AKnc_m(X)$ is a subset then we denote by $\mathcal{C}^0$ the subset of those $T\in\mathcal{C}$ satisfying $\bdry T=0$ if $m\geq 1$ and $T(1)=0$ if $m=0$.

Let $T\in\AKirc_m(X)$. Then $\set(T)$ is defined by $$\set(T):= \{x\in X: \Theta_{*m}(\|T\|, x)>0\},$$ where $\Theta_{*m}(\|T\|, x)$ is the lower $m$-density of $\|T\|$ at $x$ given by $$\Theta_{*m}(\|T\|, x):= \liminf_{r\to0^+}\frac{\|T\|(B(x,r))}{\omega_mr^m}$$ and $\omega_m$ is the volume of the unit ball in $\R^m$. It is shown in \cite[Theorem 4.6]{Ambrosio-Kirchheim-currents} that $\set(T)$ is a countably $\haus^m$-rectifiable set on which $\|T\|$ is concentrated, that is, $\|T\|(X\backslash\set(T))=0$. 

We make the following elementary but useful observation concerning Lipschitz curves and the currents which they induce. 

\begin{lemma}\label{lemma:Lip-curve-current-mass-integral}
 Given a Lipschitz curve $c:[a,b]\to X$, the integral current $T:=c_\#\Lbrack 1_{[a,b]}\Rbrack$ satisfies $\bdry T = \Lbrack c(b)\Rbrack - \Lbrack c(a)\Rbrack$ and $\mass(T)\leq\length(c)$; moreover, if $c$ is injective then $\mass(T)=\length(c)$. Finally, for every Borel function $g:X\to [0,\infty]$ we have
\begin{equation}\label{eq:integral-curve-current-Borelfct}
 \int_X g\,d\|T\| \leq \int_a^b g\circ c(t) |\dot{c}|(t)\,dt;
\end{equation}
if $\mass(T) = \length(c)$ then equality holds in \eqref{eq:integral-curve-current-Borelfct}.
\end{lemma}

\begin{proof}
 Firstly, note that $$\partial T = c_\#(\partial \Lbrack 1_{[a,b]}\Rbrack) = \Lbrack c(b)\Rbrack - \Lbrack c(a)\Rbrack.$$ Now, given Lipschitz functions $f, \pi$ on $X$ with $f$ bounded we have $$|T(f,\pi)| = \left|\int_a^b f\circ c(t) (\pi\circ c)'(t) dt\right| \leq \Lipconst(\pi) \int_a^b|f\circ c(t)| |\dot{c}|(t) dt,$$ from which it follows that
 \begin{equation}\label{eq:mass-measure-curve-current}
  \|T\| \leq c_\#(|\dot{c}| \lm^1)
 \end{equation}
 and thus $\mass(T) \leq \int_a^b|\dot{c}|(t)dt = \length(c)$ and \eqref{eq:integral-curve-current-Borelfct} for every Borel function $g:X\to [0,\infty]$. It now follows directly from \eqref{eq:mass-measure-curve-current} that if $c$ is such that $\mass(T) = \length(c)$ then we have equality in \eqref{eq:mass-measure-curve-current}. Finally, suppose $c$ is injective. Let $\varepsilon>0$ and set  $H:= \{t\in[a,b]: |\dot{c}|(t)\not=0\}$. By \cite[Lemma 4]{Kirchheim-rectifiable} there exist $\lambda_i\in(0,\infty)$ and $K_i \subset [a,b]$ compact, pairwise disjoint, and satisfying $\lm^1(H\backslash \cup K_i)=0$ and $$\lambda_i |t-s| \leq d(c(t), c(s))\leq (1+\varepsilon) \lambda_i |t-s|$$ for all $t,s\in K_i$. Set $\mu:= c_\#(|\dot{c}| \lm^1)$. Fix $i$ and let $\pi$ be a $1$-Lipschitz function on $X$ which extends $\lambda_i \left(c|_{K_i}\right)^{-1}$. It then follows that $$\|T\|(c(K_i)) \geq |T(1_{c(K_i)}, \pi)| = \lambda_i \lm^1(K_i) \geq \frac{1}{1+\varepsilon} \mu(c(K_i)).$$ Since $i$ was arbitrary, and the $c(K_i)$ are pairwise disjoint, and $\mu(X\backslash \cup c(K_i))=0$ we obtain that $$\mass(T) \geq \sum\|T\|(c(K_i)) \geq \frac{1}{1+\varepsilon} \sum \mu(c(K_i)) = \frac{1}{1+\varepsilon}\mu(\cup c(K_i)) = \frac{1}{1+\varepsilon} \mu(X).$$ Since $\varepsilon>0$ was arbitrary this yields equality in \eqref{eq:mass-measure-curve-current} and concludes the proof.
\end{proof}

As above, let $(X,d)$ be a complete metric space and endow $[0,1]\times X$ with the Euclidean product metric. Given a Lipschitz function $f$ on $[0,1]\times X$ and $t\in[0,1]$ we define the function $f_t:X\longrightarrow \R$ by
$f_t(x):= f(t,x)$. To every $T\in\AKnc_m(X)$ and every $t\in[0,1]$ we associate the normal $m$-current on $[0,1]\times X$ given
by the formula
\begin{equation*}
  ([t]\times T)(f,\pi_1,\dots,\pi_m):= T(f_{t},\pi_{1\,t},\dots,\pi_{m\,t}). 
\end{equation*}
The product of a normal current with the interval $[0,1]$ is defined by
  \begin{equation*}
  \begin{split}
   ([0,1]\times& T) (f,\pi_1,\dots,\pi_{m+1}):= \\
      &\sum_{i=1}^{m+1}(-1)^{i+1}\int_0^1T\left(f_t\frac{\partial \pi_{i\,t}}{\partial t},\pi_{1\,t},
                                               \dots,\pi_{i-1\,t},\pi_{i+1\,t},\dots,\pi_{m+1\,t}\right)dt
  \end{split}
  \end{equation*}
  for $(f,\pi_1,\dots,\pi_{m+1})\in\Lipspace_b([0,1]\times X)\times\Lipspace^{m+1}([0,1]\times X)$.
It can be proved, see \cite{Ambrosio-Kirchheim-currents} and also \cite{Wenger-isop-Eucl}, that $[0,1]\times T\in\AKnc_{m+1}([0,1]\times X)$ and
 \begin{equation*}
   \partial([0,1]\times T)= [1]\times T - [0]\times T - [0,1]\times\partial T
  \end{equation*}
  if $m\geq 1$ and $ \partial([0,1]\times T)= [1]\times T - [0]\times T$ if $m=0$; moreover, if $T\in\AKic_m(X)$ then $[0,1]\times T\in\AKic_{m+1}([0,1]\times X)$.
We have the following simple lemma which estimates the mass of the push-forward of $[0,1]\times T$ under a Lipschitz map.

\begin{lemma}\label{lemma:mass-estimate-push-cone}
Let $\psi:[0,1]\times X\to Y$ be a Lipschitz map, where $Y$ is a complete metric space, and let $T\in\AKnc_m(X)$. Suppose $\lambda:[0,1]\to[0,\infty)$ and $\delta:X\to[0,\infty)$ are bounded and Borel measurable functions such that $\psi(t,\cdot)$ is $\lambda(t)$-Lipschitz for every $t\in[0,1]$ and $\psi(\cdot, x)$ is $\delta(x)$-Lipschitz for every $x\in X$. Then we have
\begin{equation*}
 \|\psi_{\#}([0,1]\times T)\| \leq (m+1) \psi_{\#}(\lambda^m\lm^1\times \delta \|T\|).
\end{equation*}
\end{lemma}

 \begin{proof}
 Let $(f,\pi_1,\dots,\pi_{m+1})\in\Lipspace_b(Y)\times\Lipspace^{m+1}(Y)$. We compute
  \begin{equation*}
   \begin{split}
    |\psi_{\#}(&[0,1]\times T)(f,\pi_1,\dots,\pi_{m+1})| \\
       &\leq \sum_{i=1}^{m+1}\left|\int_0^1 
             T(f\circ\psi_t\frac{\partial (\pi_i\circ\psi_t)}{\partial t},\pi_1\circ\psi_t,\dots,\pi_{i-1}\circ\psi_t,
               \pi_{i+1}\circ\psi_t,\dots,\pi_{m+1}\circ\psi_t)dt\right|\\
                                        &\leq \sum_{i=1}^{m+1} \int_0^1\prod_{j\not=i}\Lipconst(\pi_j\circ\psi_t)
                                                \int_X\left|f\circ\psi_t\frac{\partial (\pi_i\circ\psi_t)}{\partial t}\right|
                                                 d\|T\|dt\\
                                        &\leq (m+1)\prod_{j=1}^{m+1}\Lipconst(\pi_j)\int_0^1
                                                \int_X|f\circ\psi(t,x)|\delta(x) d\|T\|(x)\lambda^m(t)dt,
   \end{split}
  \end{equation*}
 from which the claim follows together with the definition of mass. 
 \end{proof}

\begin{defn}
Let $m\geq 0$. Given $T\in\AKc_m(X)$ and $\mathcal{C}\subset \AKc_{m+1}(X)$ we define
$$
\fillvol(T, \mathcal{C})= \inf\{\mass(S):\, S \in \mathcal{C},\, \bdry S = T\},
$$
where we use the convention $\inf\emptyset = \infty$.
\end{defn}
 If $T\in\AKic_m(X)$ then we usually abbreviate 
$$
\fillvol(T):= \fillvol(T, \AKic_{m+1}(X)).
$$

\begin{defn}
Let $m\geq 0$ and let $\mathcal{C} = (\mathcal{C}_m, \mathcal{C}_{m+1})$ with $\mathcal{C}_k \subset\AKc_k(X)$ for $k=m, m+1$, and such that $\partial S \in \mathcal{C}_m$ for all $S\in \mathcal{C}_{m+1}$. The flat norm of an element $T\in \mathcal{C}_m$ is defined by 
\begin{equation}\label{eq:def-flatnorm-int-rect-FAT} 
 \flatnormFAT(T, \mathcal{C}):= \inf\{ \mass(R)+\mass(V):\, R \in \mathcal{C}_m,\, V \in \mathcal{C}_{m+1}, \, T=R+\bdry V  \}. 
\end{equation}
\end{defn}
It is clear that $\flatnormFAT(T, \mathcal{C})\leq \mass(T)$, moreover $\flatnormFAT(\bdry S,\mathcal{C})\leq \mass(S)$ if $S\in \mathcal{C}_{m+1}$. 
If $\mathcal{C}_k=\AKc_k(X)$ for $k=m, m+1$ we will write $\flatnormFAT(T)$ instead of $\flatnormFAT(T, \mathcal{C})$. If $\mathcal{C}_k=\AKirc_k(X)$ for $k=m, m+1$ we will write $\flatnormSCR(T)$ instead of $\flatnormFAT(T, \mathcal{C})$.
Note that for $T\in\AKirc_m(X)$ we have $\flatnormFAT(T)\leq \flatnormSCR(T)$. 
Note also that for $n\in\Z$ and $T\in\AKirc_m(X)$ we have $\flatnormSCR(nT)\leq |n|\flatnormSCR(T)$ and strict inequality can occur, see \cite{White-multiple}.
If $T\in\AKnc_m(X)$ then we have $\flatnormFAT(T)\leq \fillvol(T, \AKnc_{m+1}(X))$, and if $T\in\AKic_m(X)$ then $\flatnormSCR(T)\leq \fillvol(T)$. Moreover, we have the following fact.

\begin{lemma}
 Let $X$ be complete metric space and $m\geq 1$. Suppose there exists $r>0$ such that $$\fillvol(T)\leq \mass(T)$$ for all $T\in\AKic^0_m(X)$ with $\mass(T)< r$. Then $$\fillvol(T) = \flatnormSCR(T)$$ for all $T\in\AKic^0_m(X)$ with $\mass(T)<r$. 
\end{lemma}

Examples of spaces satisfying the hypotheses in the lemma include Banach spaces, ${\rm CAT}(\kappa)$-spaces, and Carnot groups with a left-invariant Finsler metric, see \cite{Wenger-isop-Eucl}, \cite{Wenger-flatconv}, \cite{Wenger-asymptotic-rank}.

\begin{proof}
 Let $T\in\AKic^0_m(X)$ with $\mass(T)<r$ and let $\varepsilon>0$ be such that $\mass(T) + \varepsilon <r$. Choose $R\in\AKic_m(X)$ and $V\in\AKic_{m+1}(X)$ such that $T= R+\partial V$ and $$\mass(R) + \mass(V)\leq \flatnormSCR(T) + \varepsilon.$$ Since $\partial R=0$ and $\mass(R)<r$ there exists $U\in\AKic_{m+1}(X)$ with $\partial U = R$ and $\mass(U)\leq \mass(R)$. It follows that $\partial(U+V) = T$ and hence $$\fillvol(T) \leq \mass(U) + \mass(V) \leq \mass(R) + \mass(V)\leq \flatnormSCR(T) +\varepsilon.$$ Since $\varepsilon>0$ was arbitrary this shows that $\fillvol(T)\leq \flatnormSCR(T)$. Since the opposite inequality holds for all $T\in\AKic_m(X)$ with $\bdry T = 0$ the proof is complete. 
\end{proof}


\section{Cochains, upper norms, and upper gradients}
\label{part1}

\subsection{Definition of cochains, upper norms, and upper gradients} 
\label{part11}
In this section we define cochains, our basic objects of study for the forthcoming sections. We first give a general definition of a cochain (Definition \ref{cochaindefn}) as a function from an additive subgroup of $m$-dimensional currents in complete metric spaces, without any regularity assumptions. We slightly abuse terminology here by only requiring sublinearity from the cochains instead of linearity. We then define upper norms and upper gradients of cochains. Using these notions, we can talk about the regularity of cochains in general (complete) metric measure spaces, and try to prove analytic properties for them. In particular, the cochains can be seen as a generalization of classical differential forms to non-smooth spaces; recall that a smooth $m$-form induces a linear cochain by integration over $m$-dimensional currents. 

Let $X$ be a complete metric space, $m\geq 0$, and let $\mathcal{C}$ be an additive subgroup of $\AKc_m(X)$.

\begin{defn}
\label{cochaindefn}
 A function $\omega: \mathcal{C}\to\overline{\R}$ is called cochain on $\mathcal{C}$ if $\omega(0) = 0$ and
 \begin{equation*}
  |\omega(T)| \leq |\omega(T+S)| + |\omega(S)|
 \end{equation*}
 for all $T,S\in\mathcal{C}$. If furthermore
 \begin{equation*}
  \omega(T+S) = \omega(T) + \omega(S)
 \end{equation*}
whenever each term is finite then $\omega$ is called a linear cochain.
\end{defn}

If $\omega$ is a cochain on $\mathcal{C}$ then clearly $|\omega(T)| = |\omega(-T)|$ and thus $$|\omega(T+S)|\leq |\omega(T)| + |\omega(S)|$$ for all $T,S\in\mathcal{C}$.
A basic example of a linear cochain is given as follows.

\begin{example}\label{example:lip-form-cochain}
 Let $(f,\pi)\in \Lipspace_b(X)\times\Lipspace^m(X)$. Then a linear cochain on $\AKc_m(X)$ is given by $\omega(T)= T(f,\pi).$ 
\end{example}

Further simple examples of cochains are provided by the mass $\mass$ and the flat norm $\flatnormFAT$, which are cochains on $\AKc_m(X)$, and by the flat norm $\flatnormSCR$, which is a cochain on $\AKirc_m(X)$. More generally, if $g,h:X\to[0,\infty]$ are Borel measurable functions then 
\begin{equation}\label{hgflatnorm}
\omega(T):= \inf\left\{\int_Xh\,d\|R\| + \int_Xg\,d\|V\|: R\in\AKc_m(X), V\in\AKnc_{m+1}(X), T=R+\bdry V\right\}
\end{equation} 
defines a cochain on $\AKc_m(X)$. Analogously, one obtains a cochain on $\AKirc_m(X)$ if for $T\in\AKirc_m(X)$ one takes the infimum over all $R\in\AKirc_m(X)$ and $V\in\AKic_{m+1}(X)$ with $T=R+\bdry V$ in the above equation \eqref{hgflatnorm}. 

\begin{defn}
Let $\omega$ be a cochain on $\mathcal{C}$. A Borel function $h:X\to[0,\infty]$ is called upper norm of $\omega$ if 
\begin{equation}
\label{unorm}
|\omega(T)| \leq \int_X h \, d\norm{T} 
\end{equation}
for every $T \in \mathcal{C}$.
\end{defn}


\begin{defn}
Let $\omega$ be a cochain on $\mathcal{C}$ and let $\mathcal{C}'\subset\AKc_{m+1}(X)$ be a subset. A Borel function $g: X\to[0,\infty]$ is called upper gradient of $\omega$ with respect to $\mathcal{C}'$ if 
\begin{equation}\label{ugrad}
|\omega(T)| \leq \int_X g \, d\norm{S} 
\end{equation}
for all $T\in\mathcal{C}$ and $S \in \mathcal{C}'$  such that $\bdry S =T$.  
\end{defn}

We often simply say ``$g$ is an upper gradient of $\omega$'' if $\mathcal{C}'$ is clear from the context. In Section~\ref{section:estimtes-upper-norm-grad} we will determine an upper norm and an upper gradient of the cochain given in Example~\ref{example:lip-form-cochain}. We will furthermore establish a precise relationship between upper gradients of the linear $0$-cochain induced by a Lipschitz function $f$ and the upper gradients of the function $f$, as defined in \cite{HeiKos}, \cite{Shan}.

\bigskip

Now assume that $X$ is equipped with a Borel regular measure $\mu$. Let $\Gamma \subset \AKc_m(X)$ be a family of currents and $1 \leq p < \infty$. The \emph{$p$-modulus} $M_p(\Gamma)$ in $(X,\mu)$ is defined as $\inf \int_X f^p \, d\mu$, where the infimum is taken over all Borel functions $f \geq 0$ such that $\int_X f \, d||T|| \geq 1$ for every $T \in \Gamma$. As a consequence of Lemma~\ref{lemma:Lip-curve-current-mass-integral} we obtain the following relationship between the modulus of a curve family and the modulus defined above. Let $\Gamma'$ be a family of Lipschitz curves in $X$ and let $\Gamma$ denote the family of integral currents induced by curves in $\Gamma'$, that is, $$\Gamma:=\{c_\#\Lbrack 1_{[a,b]}\Rbrack: \text{$c$ is a curve in $\Gamma'$ and parameterized on $[a,b]$}\}.$$ Then we have $M_p(\Gamma)\geq M_p(\Gamma')$, where the right hand side denotes the modulus of the curve family as defined e.g. in \cite{HeiKos}, \cite{Shan}. Moreover, if every curve in $\Gamma'$ is injective then $M_p(\Gamma)= M_p(\Gamma')$. 

The theory of $p$-modulus of general measures and Lipschitz surfaces was initiated by Fuglede  \cite{Fuglede-modulus}. Ziemer \cite{Ziemer} applied the theory of currents to prove a duality estimate between capacities and moduli of separating surfaces. Surface modulus has recently been applied in quasiconformal mapping theory, cf. \cite{Raj}, \cite{HeiWu}, \cite{PanWu}. 


Let $\Lambda \subset \AKc_m^0(X)$ and $\mathcal{C}'\subset\AKc_{m+1}(X)$, and let $1 \leq p < \infty$. We define the \emph{$p$-capacity} 
$\operatorname{cap}_p(\Lambda, \mathcal{C}')$ by 
$$
\operatorname{cap}_p(\Lambda, \mathcal{C}')= M_p(\Gamma), 
$$
where 
$$
\Gamma = \{S \in \mathcal{C}': \bdry S =T \text{ for some } T \in \Lambda \}. 
$$  
In Section~\ref{section:rel-haus-cap} we will establish a relationship between the Hausdorff dimension of a set $A$ and the capacity of a family of currents with support in $A$. In Section~\ref{section:lie-haus-cap} we will furthermore establish lower bounds for the capacity in the setting of Lie groups, endowed with a left-invariant Finsler metric.

Given $1\leq p,q\leq \infty$ we denote by $\mathcal{L}_q(\mathcal{C})$ the family of cochains on $\mathcal{C}$ which have an upper norm in $L^q(X,\mu)$ and by $\mathcal{W}_p(\mathcal{C}, \mathcal{C}')$ the family of cochains on $\mathcal{C}$ which have an upper gradient with respect to $\mathcal{C}'$ which is in $L^p(X,\mu)$. We furthermore set $$\mathcal{W}_{q,p}(\mathcal{C}, \mathcal{C}'):= \mathcal{L}_q(\mathcal{C})\cap \mathcal{W}_p(\mathcal{C}, \mathcal{C}').$$ 
If $\omega \in \mathcal{W}_{q,p}(\mathcal{C}, \mathcal{C}')$, then we denote 
\begin{equation}
\label{poli}
\norm{\omega}_{q,p}=\inf \, \norm{h}_q+\norm{g}_p, 
\end{equation}
where the infimum is taken with respect to upper norms $h$ and upper gradients $g$ of $\omega$ with respect to $\mathcal{C}'$. 

In the sequel we will use the abbreviations $\mathcal{W}_p(\AKnc_m(X)):= \mathcal{W}_p(\AKnc_m(X), \AKnc_{m+1}(X))$ and $\mathcal{W}_p(\AKic_m(X)):= \mathcal{W}_p(\AKic_m(X), \AKic_{m+1}(X))$; $\mathcal{W}_{q,p}(\AKnc_m(X)):= \mathcal{W}_{q,p}(\AKnc_m(X), \AKnc_{m+1}(X))$ and $\mathcal{W}_{q,p}(\AKic_m(X)):= \mathcal{W}_{q,p}(\AKic_m(X), \AKic_{m+1}(X))$. 
Examples of $\mathcal{W}_{q,p}$-cochains are given in Example \ref{example:lip-form-cochain} (see Proposition \ref{prop:upper-grad-norm-Lipform-cochain}). Also, it is straightforward to verify that the function $h$ in \eqref{hgflatnorm} is an upper norm of the corresponding cochain $\omega$, and $g$ is an upper gradient (notice that we can restrict to surfaces $R \in \mathcal{C}$ and $V \in \mathcal{C}'$ in \eqref{hgflatnorm}). So, if we assume $h \in L^q$ and $g \in L^p$, then $\omega \in \mathcal{W}_{q,p}$. We discuss another basic set of examples in Section \ref{exception}


\subsection{Exceptional sets and Sobolev cochains} \label{exception}
In this section we define weak versions of upper norms and upper gradients, and the (Newtonian) Sobolev spaces $W_{q,p}$ of linear cochains. We then show that Euclidean differential forms which belong to the Sobolev space $W^{q,p}_d(\R^n,\bigwedge^m)$ (see the definition below) also belong to $W_{q,p}$. 

Let $\omega:\mathcal{C}\to\overline{\R}$ be a cochain on an additive subgroup $\mathcal{C}$ of $\AKc_m(X)$.
We say that a Borel function $h:X\to[0,\infty]$ is a \emph{$q$-weak upper norm} of $\omega$, where $1\leq q<\infty$, if \eqref{unorm} holds for every $T \in\mathcal{C}\setminus \Gamma$ for some family $\Gamma\subset\mathcal{C}$ with  $M_q(\Gamma)=0$. 
Let $\mathcal{C}'\subset \AKc_{m+1}(X)$. Similarly, we say that a Borel function $g:X\to[0,\infty]$ is a \emph{$p$-weak upper gradient} of $\omega$ with respect to $\mathcal{C}'$, where $1\leq p<\infty$, if \eqref{ugrad} holds for every $S \in\mathcal{C}'\setminus \Lambda$ for some family $\Lambda\subset\mathcal{C}'$ with $M_p(\Lambda)=0$. 

It follows from the definition of modulus that, if $\Lambda \subset \AKc_m(X)$ satisfies $M_p(\Lambda)=0$, then there exists a Borel function $f \in L^p(X,\mu)$ such that $\int_X f \, d\norm{T}=\infty$ for every $T \in \Lambda$. Therefore, a cochain $\omega$ has a $p$-integrable upper gradient (upper norm) if and only if it has a $p$-weak upper gradient (upper norm). 

The following lemma is a special case of  \cite[Theorem 3]{Fuglede-modulus}.

\begin{lemma}[Fuglede's lemma]
\label{fugledelemma}
Let $1 \leq p < \infty$, and let $f$ be a Borel function. Moreover, let $(f_j)$ be a sequence of Borel functions converging to $f$ in $L^p(X,\mu)$. Then there exist a subsequence 
$(f_{j_k})$ and $\Lambda \subset \AKc_m(X)$ with $M_p(\Lambda)=0$ such that 
$$
\int_{X}  |f_{j_k}-f| \, d\norm{T} \to 0 
$$
for every $T \in \AKc_m(X) \setminus \Lambda$. 
\end{lemma}

Suppose now that $\omega \in \mathcal{L}_q(\mathcal{C})$, $1<q< \infty$, and let $(h_j)$ be a sequence of upper norms of $\omega$ such that 
$$
\lim_{j \to \infty} \int_X h_j^q \, d\mu = \inf_{h} \int_X h^q \, d\mu, 
$$
where the infimum is taken over all upper norms $h$ of $\omega$. By weak compactness, there is a subsequence, also denoted by $(h_j)$, converging weakly in $L^q$ to 
$h_0 \in L^q(X,\mu)$. Moreover, by Mazur's lemma, there is a sequence of convex combinations $\tilde{h}_k$ of the functions $h_j$ converging strongly in $L^q$ to $h_0$. Clearly, each $\tilde{h}_k$ is also an upper norm of $\omega$, so by Lemma \ref{fugledelemma}, $h_0$ is a $q$-weak upper norm of $\omega$. Similarly, we see that $L^p$-bounded sequences of upper gradients converge, up to a subsequence, to a $p$-weak upper gradient. It follows in particular that when $1< p,q< \infty$, the infimum in \eqref{poli} is attained by some $q$-weak upper norm $h_0$ and $p$-weak upper gradient $g_0$.

We now turn to the definition of the Sobolev space of linear cochains. 

\begin{lemma}
\label{sobolemma}
Let $\omega_1, \omega_2: \mathcal{C} \to \overline{\R}$ be linear cochains. Define $\omega_1+\omega_2$ by setting 
$$
(\omega_1+\omega_2)(T)=\omega_1(T)+\omega_2(T) \quad \text{if} \quad |\omega_1(T)|+|\omega_2(T)| < \infty, 
$$ 
and $(\omega_1+\omega_2)(T)=\infty$ otherwise. Then $\omega_1+\omega_2$ is a linear cochain on $\mathcal{C}$. Moreover, if $1 \leq q,p < \infty$ and $\omega_1,\omega_2 \in \mathcal{W}_{q,p}(\mathcal{C},\mathcal{C'}) $, then also $\omega_1+\omega_2 \in \mathcal{W}_{q,p}(\mathcal{C},\mathcal{C'})$.   
\end{lemma}


\begin{proof}
Let $T,S \in \mathcal{C}$. Firstly, if 
$$
|(\omega_1+\omega_2)(T+S)|+|(\omega_1+\omega_2)(T)|+|(\omega_1+\omega_2)(S)| < \infty,
$$ 
then also $|\omega_i(T+S)|+|\omega_i(T)|+|\omega_i(S)| < \infty$ for $i=1,2$, and so 
$$
(\omega_1+\omega_2)(T+S)=(\omega_1+\omega_2)(T)+(\omega_1+\omega_2)(S).
$$ 
Secondly, if $|(\omega_1+\omega_2)(T+S)|=\infty$, then the definition of cochain implies that $|\omega_i(T)|+|\omega_i(S)|=\infty$ for $i=1$ or $i=2$. If follows that 
$|(\omega_1+\omega_2)(T)|+|(\omega_1+\omega_2)(S)|=\infty$. We conclude that $\omega_1+\omega_2$ satisfies the conditions of a linear cochain. Also, if $h_1$, $h_2$ are upper norms and $g_1$ and $g_2$ are upper gradients with respect to $\mathcal{C}'$ of $\omega_1$ and $\omega_2$, respectively, then $h_1+h_2$ and $g_1+g_2$ are upper norm and upper gradient, with respect to $\mathcal{C}'$, of $\omega_1+\omega_2$. 
\end{proof}

It is clear that $\lambda \omega$ belongs to $\mathcal{W}_{q,p}(\mathcal{C},\mathcal{C'})$ for every $\lambda \in \R$ if $\omega$ does. Therefore, Lemma \ref{sobolemma} implies that 
the set of linear cochains in $\mathcal{W}_{q,p}(\mathcal{C},\mathcal{C'})$ forms a vector space. We equip this space with the seminorm $\norm{\omega}_{q,p}$ defined in \eqref{poli}. 

\begin{defn}
\label{sobospace}
The space $W_{q,p}(\mathcal{C},\mathcal{C'})$ is the set of equivalence classes of linear cochains in $\mathcal{W}_{q,p}(\mathcal{C},\mathcal{C'})$ under the equivalence relation defined by $\omega_1 \sim \omega_2$ if $\norm{\omega_1-\omega_2}_{q,p}=0$. 
\end{defn} 

We see that $W_{q,p}(\mathcal{C},\mathcal{C'})$ equipped with the norm $\norm{\cdot}_{q,p}$ is a normed space. Moreover, if $1<p,q<\infty$, and if $\omega_1$ and $\omega_2$ are cochains representing the same element in  $W_{q,p}(\mathcal{C},\mathcal{C'})$, then 
$\omega_1(T)=\omega_2(T)$ for every $T \in \mathcal{C} \setminus (\Gamma \cup \Lambda)$, where $M_q(\Gamma)=\operatorname{cap}_p(\Lambda, \mathcal{C}')=0$. Following the proof of \cite[Theorem 3.7]{Shan}, one can show that $W_{q,p}(\mathcal{C},\mathcal{C'})$ is a Banach space. We do not develop further properties of the Sobolev spaces here.

We next show that Sobolev forms in the space $W^{q,p}_d(\R^n,\bigwedge^m)$ induce cochains in the space 
$W_{q,p}(\AKc_m(\R^n),\AKc_{m+1}(\R^n))$. Let $1<q,p<\infty$, and let $\omega$ be a differential $m$-form expressed in Euclidean coordinates by 
$$
\omega= \sum_I \omega_I \, dx_I. 
$$
We assume that the coefficients $\omega_I$ belong to $L^q(\R^n)$. Furthermore, we say that the $(m+1)$-form $d\omega=\sum_{J}(d\omega)_J \, dx_J$ is the distributional exterior derivative of $\omega$ if 
$$
\int_{\R^n} d \omega \wedge \varphi  = (-1)^{m+1} \int_{\R^n} \omega \wedge d \varphi 
$$
for every smooth, compactly supported $(n-m-1)$-form $\varphi$. We assume that the coefficients $(d\omega)_J$ belong to $L^p(\R^n)$. Then we say that $\omega$ belongs to the Sobolev space $W^{q,p}_d(\R^n,\bigwedge^m)$. See \cite{IwaLuto} and \cite{Iwaniecetal} for the $L^p$-theory of differential forms. 
Let $\omega \in W^{q,p}_d(\R^n,\bigwedge^m)$. Then there is a sequence of smooth compactly supported $m$-forms $\omega^j$ converging to $\omega$ in $W^{q,p}_d(\R^n,\bigwedge^m)$, i.e. 
$$
\sum_I \int_{\R^n} |\omega^j_I-\omega_I|^q \, dx + \sum_J \int_{\R^n} |(d\omega^j)_J-(d\omega)_J|^p \, dx \to 0 
$$
as $j \to \infty$. 
Let $T \in \AKc_m(\R^n)$, and define 
$$
\tilde{\omega}^j(T)= \sum_I T(\omega^j_I,x_{i_1},\ldots,x_{i_m}), 
$$
where $dx_{i_1}\wedge \ldots \wedge dx_{i_m}=dx_I$. Then $\tilde{\omega}^j$ is a linear cochain, and 
$$
|\tilde{\omega}^j(T)| \leq C_1 \int_{\R^n} |\omega^j|\, d\norm{T}, 
$$
where $C_1$ depends only on $n$, and $|\omega^j|$ is the euclidean norm of the coefficients $\omega^j_I$. We conclude that $C_1|\omega^j|$ is an upper norm of $\tilde{\omega}^j$. 
Next, for $S \in \AKc_{m+1}(\R^n)$, define 
$$
\tilde{d\omega}^j(S)=\sum_J S(d\omega^j_J, x_{\ell_1},\ldots,x_{\ell_{m+1}}),    
$$
where $dx_{\ell_1}\wedge \ldots \wedge dx_{\ell_{m+1}}=dx_J$. Then, if $\partial S =T$, 
$$
\tilde{\omega}^j(T)=\tilde{d \omega}^j(S); 
$$
this can be seen by approximating the coefficients $\omega_I^j$ by polynomials and applying the product rule and the alternating properties of currents, see \cite{Ambrosio-Kirchheim-currents}. We conclude that 
$$
|\tilde{\omega}^j(T)| \leq C_2 \int_{\R^n} |d\omega^j| \, d\norm{S}, 
$$
where $C_2$ depends only on $n$, and $|d\omega^j|$ is the euclidean norm of the coefficients $d \omega^j_J$. We conclude that $C_2|d\omega^j|$ is an upper gradient of $\tilde{\omega}^j$. 
By Lemma \ref{fugledelemma}, there is a subsequence, also denoted by $(\omega^j)$, such that 
$$
 \int_{\R^n} |\omega^j_I-\omega_I| \, d\norm{T} \to 0 
$$
for every $T \in \AKc_m \setminus \Gamma$, where $M_q(\Gamma)=0$, and 
$$
 \int_{\R^n} |(d\omega^j)_J-(d\omega)_J| \, d\norm{S} \to 0 
$$
for every $S \in \AKc_{m+1} \setminus \Lambda$, where $M_p(\Lambda)=0$. We define $\tilde{\omega}:\AKc_m \to \R \cup \{\infty\}$ by $\tilde{\omega}(T):= \lim_{j \to \infty} \tilde{\omega}^j(T)$ when the limit exists, and $\infty$ otherwise. We see that 
$\tilde{\omega}$ is a linear cochain in the sense of Definition \ref{cochaindefn}. Furthermore, 
$$
|\tilde{\omega}(T)| \leq C_1 \lim_{j \to \infty} \int_{\R^n} |\omega^j| \, d\norm{T} = C_1 \int_{\R^n} |\omega|\, d\norm{T} 
$$
for every $T \in \AKc_m \setminus \Gamma$, where $M_q(\Gamma)=0$, so $C_1 |\omega|$ is a $q$-weak upper norm of $\tilde{\omega}$. Similarly, there is a set $\Lambda \subset \AKc_{m+1}$ of zero $p$-modulus such that  
$$
|\tilde{\omega}(T)| \leq C_2 \lim_{j \to \infty} \int_{\R^n} |d\omega^j| \, d\norm{S} = C_2 \int_{\R^n} |d\omega|\, d\norm{S} 
$$
whenever $S \in \AKc_{m+1} \setminus \Lambda$, $\partial S=T$, so $C_2 |d\omega|$ is a $p$-weak upper gradient of $\tilde{\omega}$ with respect to $\AKc_{m+1}$. This shows that $\omega$ induces a cochain $\tilde{\omega} \in W_{q,p}(\AKc_m(\R^n),\AKc_{m+1}(\R^n))$. Moreover, the corresponding Sobolev norms are equivalent.

\subsection{Estimates for upper norm and upper gradient}\label{section:estimtes-upper-norm-grad}

In this section we prove several results concerning upper norms and upper gradients of the cochain defined in Example~\ref{example:lip-form-cochain}. For this we first recall that for a Lipschitz function $f:X\to\R$, defined on a metric space $X$, the pointwise Lipschitz constants of $f$ are defined by
\begin{equation*}
 \ptLip f(x):= \lim_{r\to 0^+} L_rf(x)
\end{equation*}
and
\begin{equation*}
\ptlowlip f(x):= \lim_{r\to0^+} \ell_rf(x),
\end{equation*}
where
\begin{equation*}
 L_rf(x):= \sup_{s<r} \sup_{d(x,y)<s}\frac{|f(x)-f(y)|}{s},
\end{equation*}
\begin{equation*}
 \ell_rf(x):= \inf_{s<r} \sup_{d(x,y)<s}\frac{|f(x)-f(y)|}{s},
\end{equation*}
see \cite{Keith-diff}.
Since $L_rf$ and $\ell_rf$ are Borel measurable (see \cite{Keith-diff}) it follows that $\ptLip f$ and $\ptlowlip f$ are Borel measurable. We can give a first estimate for upper norm and upper gradient of the above mentioned cochain as follows.

\begin{proposition}\label{prop:upper-grad-norm-Lipform-cochain}
 Let $X$ be a complete metric space, $m\geq 0$, and $(f,\pi_1,\dots,\pi_m)\in \Lipspace_b(X)\times\Lipspace^m(X)$. Define $\omega:\AKc_m(X)\to\R$ by $\omega(T):= T(f,\pi_1,\dots,\pi_m)$. Then
 \begin{equation*}
  h(x):= |f(x)|\prod_{i=1}^m\ptLip \pi_i(x)
 \end{equation*}
 is an upper norm of $\omega$ and
 \begin{equation*}
  g(x):= \ptLip f(x)\prod_{i=1}^m\ptLip \pi_i(x) 
 \end{equation*}
 is an upper gradient of $\omega$ with respect to $\AKc_{m+1}(X)$. 
 \end{proposition}

If the cochain defined in  Proposition~\ref{prop:upper-grad-norm-Lipform-cochain} is restricted to $\AKic_m(X)$ then $\ptLip \pi_i$ can be replaced by $\ptlowlip \pi_i$. More precisely, we have the following.

\begin{proposition}\label{prop:upper-grad-norm-lowlip}
Let $X$ be a complete metric space, $m\geq 0$, and $(f,\pi_1,\dots,\pi_m)\in \Lipspace_b(X)\times\Lipspace^m(X)$. Define $\omega:\AKic_m(X)\to\R$ by $\omega(T):= T(f,\pi_1,\dots,\pi_m)$. Then
 \begin{equation*}
  h(x):= |f(x)|\prod_{i=1}^m\ptlowlip \pi_i(x)
 \end{equation*}
 is an upper norm of $\omega$ and
 \begin{equation*}
  g(x):= \ptlowlip f(x)\prod_{i=1}^m\ptlowlip \pi_i(x) 
 \end{equation*}
 is an upper gradient of $\omega$ with respect to $\AKic_{m+1}(X)$. 
\end{proposition}

In both propositions above, if $m=0$ then the products $\prod_{i=1}^m\ptLip \pi_i(x)$ and $\prod_{i=1}^m\ptlowlip \pi_i(x)$ appearing in the definitions of $h$ and $g$ should be replaced by $1$. Proposition~\ref{prop:upper-grad-norm-lowlip} provides an analog for cochains of the fact, proved by Cheeger in \cite{Cheeger-diff}, that if $f:X\to\R$ is a Lipschitz function then $\ptlowlip f(\cdot)$ is an upper gradient of $f$. Actually, this fact also follows from Proposition~\ref{prop:upper-grad-norm-lowlip} above together with Proposition~\ref{prop:upper-grad-lipfct-0-cochain} below.

Propositions~\ref{prop:upper-grad-norm-Lipform-cochain} and \ref{prop:upper-grad-norm-lowlip} come as a direct consequence of the following lemma.

\begin{lemma}
Let $X$ be a complete metric space, $m\geq 1$, and $T\in\AKc_m(X)$. Then for every bounded Borel function $f$ on $X$ and Lipschitz functions $\pi_1, \dots, \pi_m$ on $X$, we have
\begin{equation}\label{eq:strong-mass-def-estimate-ptLip}
 |T(f,\pi_1,\dots,\pi_m)|\leq \int_X|f(x)| \prod_{i=1}^m\ptLip \pi_i(x) \;d\norm{T}(x);
\end{equation}
if $T\in\AKirc_m(X)$ then we have
\begin{equation}\label{eq:strong-mass-def-estimate-rect-ptlowlip}
 |T(f,\pi_1,\dots,\pi_m)|\leq \int_X|f(x)| \prod_{i=1}^m\ptlowlip \pi_i(x) \;d\norm{T}(x).
\end{equation}
\end{lemma}

\begin{proof}
We first prove \eqref{eq:strong-mass-def-estimate-ptLip}. For this, it suffices to show that for any $m\geq 1$, any $T\in\AKc_m(X)$, and any $\tau:X\to\R$ Lipschitz
\begin{equation}\label{eq:measure-restriction-ptLip}
 \norm{T\rstr(1,\tau)} \leq \ptLip\tau(\cdot)\,\norm{T}.
\end{equation}
Indeed, for Lipschitz functions $\pi_1, \dots, \pi_m$ on $X$, successive application of \eqref{eq:measure-restriction-ptLip} together with the fact that
\begin{equation*}
 T\rstr(1,\pi_1,\dots, \pi_{k+1}) = (T\rstr(1,\pi_1,\dots, \pi_k))\rstr(1,\pi_{k+1})
\end{equation*}
yields
 \begin{equation*}
  \norm{T\rstr(1,\pi_1,\dots,\pi_m)} \leq \prod_{i=1}^m\ptLip \pi_i(\cdot) \;\norm{T}
 \end{equation*}
and hence
 \begin{equation*}
 \begin{split}
  |T(f, \pi_1,\dots,\pi_m)| &\leq \int_X|f(x)|\, d\norm{T\rstr(1,\pi_1,\dots,\pi_m)}(x)\\
   &\leq  \int_X|f(x)| \prod_{i=1}^m\ptLip \pi_i(x) \;d\norm{T}(x).
   \end{split}
 \end{equation*}
 In order to prove \eqref{eq:measure-restriction-ptLip}, let $r, \varepsilon>0$. Since $\spt T$ is $\sigma$-compact there exists a countable family $(B_i)_{i\in\N}$ of pairwise disjoint Borel sets in $X$ of diameter strictly smaller than $r$ such that $\spt T\subset\cup_iB_i$. Let $\tau\in\Lipspace(X)$ and define for $j\in\N$,
 \begin{equation*}
  A_j:= \{x\in X:  \varepsilon(j-1)\leq L_r\tau(x)< \varepsilon j\}.
 \end{equation*}
 Note that the $A_j$ are Borel sets and pairwise disjoint. It is clear that $\tau|_{B_i\cap A_j}$ is $\varepsilon j$-Lipschitz. By Mc-Shane's extension theorem there thus exists an $\varepsilon j$-Lipschitz extension $\bar{\tau}_{i,j}$ of $\tau|_{B_i\cap A_j}$ to all of $X$. Given  $(f,\pi_1,\dots,\pi_{m-1})\in \Lipspace_b(X)\times\Lipspace^{m-1}(X)$ with $\Lipconst(\pi_k)\leq 1$ for all $k$, it follows from the strengthened locality property \cite[Theorem 3.5 (iii)]{Ambrosio-Kirchheim-currents} that
\begin{equation*}
 \begin{split}
 |T(f1_{B_i\cap A_j}, \tau, \pi_1,\dots,\pi_{m-1})| &= |T(f1_{B_i\cap A_j}, \bar{\tau}_{i,j}, \pi_1,\dots,\pi_{m-1})|\\
  &\leq \Lipconst(\bar{\tau}_{i,j}) \int_{B_i\cap A_j}|f|\;d\norm{T}\\
  \end{split}
\end{equation*}
and thus
\begin{equation*}
 \begin{split}
   |T(f,\tau, \pi_1,\dots, \pi_{m-1})| &\leq \sum_j\sum_i |T(f1_{B_i\cap A_j}, \tau, \pi_1,\dots, \pi_{m-1})|\\
    &\leq \sum_j \int_{A_j}|f(x)|\,(L_r\tau(x) +\varepsilon)\;d\norm{T}(x)\\
    &= \int_X|f(x)|\,L_r\tau(x)\;d\norm{T}(x) + \varepsilon\int_X|f|\;d\norm{T}. 
  \end{split}
\end{equation*}
Since $r, \varepsilon>0$ were arbitrary it follows together with dominated convergence that 
\begin{equation*}
 |T(f,\tau, \pi_1,\dots, \pi_{m-1})| \leq \int_X |f(x)|\ptLip\tau(x)\;d\norm{T}(x),
\end{equation*}
which proves \eqref{eq:measure-restriction-ptLip} and thus \eqref{eq:strong-mass-def-estimate-ptLip}.

We now prove \eqref{eq:strong-mass-def-estimate-rect-ptlowlip}. For this, suppose $T\in\AKirc_m(X)$. By \cite[Theorem 4.5]{Ambrosio-Kirchheim-currents} we may assume without loss of generality that $T= \varphi_{\#}\Lbrack \theta\Rbrack$ for some biLipschitz map $\varphi: K\to X$ with $K\subset\R^m$ compact and $\theta\in L^1(K,\Z)$. 
View $X$ as a subset of $\ell^\infty(X)$ and let $\bar{\varphi}:\R^m\to \ell^\infty(X)$ be a Lipschitz extension of $\varphi$. Set $\pi:= (\pi_1,\dots,\pi_m)$ and let $\bar{\pi}: \ell^\infty(X)\to\R^m$ be a Lipschitz extension of $\pi$. It follows from \cite{Kirchheim-rectifiable} that for almost every Lebesgue point $x\in K$ the metric derivative $$\md\bar{\varphi}_x(v):= \lim_{r\to 0}\frac{d(\bar{\varphi}(x+rv), \bar{\varphi}(x))}{r}$$  exists for all $v\in\R^m$, is a norm on $\R^m$, and is independent of the choice of extension. We can therefore write $\md\varphi_x$ instead of $\md\bar{\varphi}_x$. By the classical Rademacher theorem $\bar{\pi}\circ\bar{\varphi}$ is differentiable at almost every Lebesgue point $x\in K$ and is independent of the choice of extensions. We can therefore write $d_x(\pi\circ\varphi)$ instead of $d_x(\bar{\pi}\circ\bar{\varphi})$. We thus obtain from an easy computation that for almost every $x\in K$
\begin{equation*}
 |\det(d_x(\pi\circ\varphi))| \leq \mathbf{J}_m^*(\operatorname{md}\varphi_x) \prod_{i=1}^m\ptlowlip \pi_i(\varphi(x)),
\end{equation*}
where $\mathbf{J}_m^*(\operatorname{md}\varphi_x)$ is given by
\begin{equation*}
 \mathbf{J}_m^*(\operatorname{md}\varphi_x):= \sup\left\{\det((L_1,\dots,L_m)): \text{ $L_i: (\R^m, \md\varphi_x)\to\R$ linear, $1$-Lip.}\right\}.
\end{equation*}
It follows that
\begin{equation*}
\begin{split}
 |T(f,\pi_1,\dots,\pi_m)| &= \left| \int_K \theta f\circ\varphi  \det(d_x(\pi\circ\varphi)) d\lm^m \right|\\
  &\leq \int_K |f\circ\varphi|\left(\prod_{i=1}^m\ptlowlip \pi_i\right)\circ\varphi\, |\theta| \mathbf{J}_m^*(\operatorname{md}\varphi) \,d\lm^m\\
  &= \int_X |f|\prod_{i=1}^m\ptlowlip \pi_i\,d\|T\|.
  \end{split}
\end{equation*}
For the last inequality we used the fact, see \cite[Theorem 9.5]{Ambrosio-Kirchheim-currents},  that $$\|T\| = \varphi_{\#}(|\theta| \mathbf{J}_m^*(\operatorname{md}\varphi) \lm^m).$$
This proves \eqref{eq:strong-mass-def-estimate-rect-ptlowlip} and completes the proof.
\end{proof}

The next result shows that upper gradients of $0$-cochains are exactly upper gradients of functions. 

\begin{proposition}\label{prop:upper-grad-lipfct-0-cochain}
 Let $X$ be a complete metric space and $f:X\to\overline{\R}$ a function. Let $\omega: \AKic_0(X)\to\overline{\R}$ be given by $$\omega(T):= \left\{\begin{array}{ll}T(f) & \spt T \subset \{|f|<\infty\} \\ +\infty &\text{otherwise.}\end{array}\right.$$ Then $\omega$ defines a linear cochain on $\AKic_0(X)$ and a Borel function $g:X\to[0,\infty]$ is an upper gradient of $\omega$ with respect to $\AKic_1(X)$ if and only if $g$ is an upper gradient of $f$.
 \end{proposition}

For the proof of the proposition we need the following weak structure result for integral $1$-currents with non-trivial boundary. 

\begin{lemma}\label{lem:weak-structure-1-currents}
 Let $X$ be a complete metric space and $T\in\AKic_1(X)$ with $\bdry T\not=0$.  Then there exist Lipschitz curves $c_i: [0,1]\to X$, $i=1,\dots, N$, where $N=\mass(\partial T)/2$, such that  $$S:= T - \sum_{i=1}^N c_{i\#}\Lbrack 1_{[0,1]}\Rbrack$$ satisfies $\partial S=0$ and 
 \begin{equation}\label{eq:opt-mass-decomp-1-currents}
  \mass(T) = \mass(S) + \sum_{i=1}^N \length(c_i).
 \end{equation}
 In particular, the curves $c_i$ satisfy $\mass(c_{i\#}\Lbrack 1_{[0,1]}\Rbrack) = \length(c_i)$. 
\end{lemma}


\begin{proof}
  Let $\hat{X}$ be a complete metric space which is a length space and which contains $X$ isometrically. Let $x_1,\dots, x_N, y_1,\dots, y_N\in X$ be points such that $$\partial T = \sum_{i=1}^N \Lbrack y_i\Rbrack - \sum_{i=1}^N \Lbrack x_i\Rbrack,$$ where $N=\mass(\bdry T)/2$.
 After possibly reindexing the $y_i$ there exist, by \cite[Lemma 4.4]{Ambrosio-Wenger}, Lipschitz curves $c_i^n:[0,1]\to \hat{X}$ with fixed Lipschitz constant and image in the closed $\frac{1}{2^n}$-neighbor\-hood $N(\spt T, 1/2^n)$ of $\spt T$, where $i=1,\dots, N$ and $n\geq 1$, such that $c_i^n(0) = x_i$, $c_i^n(1)=y_i$, and such that $S_n:= T- \sum_{i=1}^N c^n_{i\#}\Lbrack 1_{[0,1]}\Rbrack$ satisfies
 $$\mass(S_n) + \sum_{i=1}^N \length(c_i^n) \leq \mass(T) + \frac{1}{2^n}.$$ 
 Note that $\partial S_n = 0$ for every $n$.
 We now claim that $$\haus^1\left(c_i^n([0,1])\backslash \set(T)\right) \leq \frac{1}{2^n}$$ for every $n\geq 1$ and every $i=1,\dots, N$. Indeed, we compute
 \begin{equation*}
  \begin{split}
   \mass(T) &=\|T\|(\set(T))\\
       &\leq \mass(S_n) + \sum_{i=1}^N\length(c_i^n) -  \sum_{i=1}^N \length(c^n_i|_{\{t:c_i^n(t)\not\in\set(T)\}}) \\
       &\leq \mass(T) + \frac{1}{2^n} - \sum_{i=1}^N\haus^1(c_i^n([0,1])\backslash \set(T)),\\
  \end{split}
 \end{equation*} 
 which establishes the claim. It now follows that for every $i$ we have $$\haus^1\left(\bigcup_{n=1}^\infty c_i^n([0,1])\right)\leq \haus^1(\set(T)) + 1 \leq \mass(T) + 1<\infty.$$ This in turn is easily seen to imply that, after possibly passing to a subsequence, for each $i$ the sequence $(c_i^n)$ converges uniformly to a Lipschitz curve $c_i: [0,1]\to X$. Set $S:= T - \sum_{i=1}^N c_{i\#}\Lbrack 1_{[0,1]}\Rbrack$ and note that $\partial S=0$ and that $S_n$ converges weakly to $S$; hence $$\mass(S) + \sum_{i=1}^N \length(c_i) \leq \liminf_{n\to\infty}\mass(S_n) + \sum_{i=1}^N \liminf_{n\to\infty}\length(c^n_i)\leq \mass(T).$$ This completes the proof.
\end{proof}

We can now prove Proposition~\ref{prop:upper-grad-lipfct-0-cochain} as follows.

\begin{proof}[Proof of Proposition~\ref{prop:upper-grad-lipfct-0-cochain}]
We first note that $T(f)$ is well-defined for any function $f:X\to\overline{\R}$ and any $T\in\AKic_0(X)$ such that $\spt T\subset\{|f|<\infty\}$ because of the special form \eqref{eq:int-curr-0-dim} of $0$-dimensional integer rectifiable currents. It follows that $\omega$ is well-defined; it is furthermore clear that $\omega$ defines a linear cochain on $\AKic_0(X)$. 
Now, suppose that $g$ is an upper gradient of $\omega$ with respect to $\AKic_1(X)$. Let $\gamma:[a,b]\to X$ be a rectifiable curve, parameterized by arc-length. Define $T\in\AKic_0(X)$ by $T:= \Lbrack\gamma(b)\Rbrack - \Lbrack\gamma(a)\Rbrack$. It follows that $\partial \gamma_{\#}\Lbrack1_{[a,b]}\Rbrack = T$ and hence $$|f(\gamma(b)) - f(\gamma(a))| = |\omega(T)| \leq \int_X g\,d\|\gamma_{\#}\Lbrack1_{[a,b]}\Rbrack\| \leq \int_a^bg\circ\gamma,$$ where the second inequality is a consequence of Lemma~\ref{lemma:Lip-curve-current-mass-integral} and where we interpret $|f(\gamma(b)) - f(\gamma(a))|$ as $\infty$ in case $|f(\gamma(a))|=\infty$ or $|f(\gamma(b))|=\infty$. This shows that $g$ is an upper gradient of $f$. Suppose now that $g$ is an upper gradient of $f$ and let $T\in\AKic_1(X)$ with $\bdry T\not=0$. Let $c_i: [0,1]\to X$, $i=1,\dots, N$, where $N=\mass(\partial T)/2$, be Lipschitz curves as in Lemma~\ref{lem:weak-structure-1-currents} and set  $$S:= T - \sum_{i=1}^N c_{i\#}\Lbrack 1_{[0,1]}\Rbrack.$$  Note that $\bdry S = 0$ and thus $$\bdry T = \sum_{i=1}^N \Lbrack c_i(1)\Rbrack - \Lbrack c_i(0)\Rbrack.$$ Since $\mass(c_{i\#}\Lbrack 1_{[0,1]}\Rbrack)=\length(c_i)$ it follows furthermore from Lemma~\ref{lemma:Lip-curve-current-mass-integral} that $\|c_{i\#}\Lbrack 1_{[0,1]}\Rbrack\| = c_{i\#}(|\dot{c}_i| \lm^1)$ and hence
 \begin{equation*}
 \begin{split}
   |\omega(\partial T)| &= \left|\sum_{i=1}^N f(c_i(1)) - f(c_i(0))\right|\\
   & \leq \sum_{i=1}^N |f(c_{i}(1)) - f(c_i(0))|\\
   & \leq \sum_{i=1}^N \int_0^1 g\circ c_i |\dot{c}_i| \\
    &= \sum_{i=1}^N \int g d\|c_{i\#}\Lbrack 1_{[0,1]}\Rbrack\|\\
    &\leq \int g\,d\|T\|.
  \end{split}
 \end{equation*}
 This also holds in the case that $|\omega(\bdry T)|=\infty$.
 Since $T$ was arbitrary this shows that $g$ is indeed an upper gradient of $\omega$ with respect to $\AKic_1(X)$. This completes the proof.
\end{proof}


\subsection{Relationship between Hausdorff measure and capacity}\label{section:rel-haus-cap}

The aim of this section is to prove the following result which gives a relationship between the Hausdorff dimension of a set and the capacity of families of currents supported on this set. We will prove further capacity results in the setting of Lie groups in 
Section~\ref{section:lie-haus-cap}. 

\begin{theorem}\label{thm:hausdim-capacity-currents}
 Let $(X, d)$ be a complete metric space and $\mu$ a Borel measure on $X$.  Let $1<p \leq Q<\infty$ and $m\geq 0$. Suppose that 
  \begin{equation*}
  \mu(B(x,r))\leq Cr^Q
 \end{equation*}
 for all $x\in X$ and all $r>0$, where $C>0$ is some fixed number. If  $A\subset X$ is a compact set with $\haus^{Q-p}(A)<\infty$ and if $\Lambda\subset\AKic_m^0(X)$ is a family with $\spt T \subset A$ for every $T\in\Lambda$ then $$\capa_p(\Lambda, \AKic_{m+1}(X))=0.$$
\end{theorem}

We remark that Theorem~\ref{thm:hausdim-capacity-currents} also holds for $p=1$ if one assumes that the Hausdorff dimension of $A$ satisfies $\dim_{\haus}(A)<Q-1$, see the proof. The proof of Theorem~\ref{thm:hausdim-capacity-currents} is a direct consequence of Propositions~\ref{prop:intersect-path-modulus-hausdorff} and \ref{prop:capacity-intcurr-vs-modulus-curves} below.

Let $E \subset \R^n$ be a set, and let $\Lambda_m(E)$ be the family of all $m$-dimensional Lipschitz surfaces intersecting $E$. Fuglede \cite[II.3]{Fuglede-modulus} has given both necessary and sufficient conditions for the $p$-modulus of $\Lambda_m(E)$ to be zero. His conditions are expressed in terms of capacities of $E$ and, as Fuglede notes, they can be translated to conditions on the Hausdorff dimension of $E$ using the relationship between capacities and Hausdorff dimensions. 

\begin{proposition}\label{prop:intersect-path-modulus-hausdorff}
Let $(X,d)$ be a metric space and $\mu$ a Borel measure on $X$. Let $1<p \leq Q<\infty$. Suppose 
\begin{equation}
\label{measbound}
\mu(B(x,r)) \leq C r^Q 
\end{equation}
for all $x \in X$ and $r>0$, where $C >0$ is some fixed number. If $A \subset X$ is a compact set with $\haus^{Q-p}(A) < \infty$ and if $\Gamma_A$ is the family of all (nonconstant) rectifiable paths intersecting $A$, then $M_p(\Gamma_A)=0$. 
\end{proposition}

\begin{proof}
Notice that our assumptions imply $\mu(A)=0$, so that the family of paths inside $A$ has zero $p$-modulus. Thus, by subadditivity of modulus, it suffices to show that 
$M_p (\Gamma_{R})=0$ for every $R>0$, 
where 
$$
\Gamma_R=\{ \gamma \in \Gamma_A:\, |\gamma| \cap X \setminus N(A,R) \neq \emptyset\},
$$
where $|\gamma|$ denotes the image of $\gamma$ and where $N(A,R)$ is defined at the beginning of Section~\ref{sec:notation}.
Fix $R>0$. Let $0<r<R$. Then we find a cover of $A$ by open balls $B(x_j,r_j)$, such that $2r_j \leq r$ for every $j$, and 
\begin{equation}
\label{kkuu}
\sum r_j^{Q-p} \leq C' \haus^{Q-p}(A)+1
\end{equation}
for a constant $C'$ only depending on $Q-p$.
By compactness of $A$, we may choose the cover to be finite; $j=1, \ldots, M(r)$. Denote 
$$
D_{r}=\bigcup_{j=1}^{M(r)} B(x_j,r_j). 
$$
We define $\rho_r$ as follows: 
$$
\rho_r(x)= \max_{j=1, \ldots, M(r)} r_j^{-1} 1_{B(x_j,2r_j)\setminus B(x_j,r_j)}(x)
$$
when $x \in X \setminus D_{r}$, and $\rho_r(x)=0$ otherwise. Then $\rho_r$ is an admissible function for the family $\Gamma_R$. 
Using \eqref{measbound} and \eqref{kkuu}, we have 
$$
\int \rho_r^p \, d\mu \leq \sum_{j=1}^{M(r)} r_j^{-p} \mu(B(x_j,2r_j)) \leq C2^Q \sum_{j=1}^{M(r)} r_j^{Q-p} 
\leq CC' 2^Q(\haus^{Q-p}(A)+1). 
$$

We now define a sequence of positive numbers $R=r_0 > 2r_1 > ...$ inductively. Assume $r_k$ is defined. Then we find a cover of $A$ with $r=r_k$ as above. By compactness of $A$ we can choose $r_{k+1}< 2r_k$ such that $N(A,r_{k+1}) \subset D_{r_k}$. 

Next, applying the above with $r=r_k$, we see that for each $k$ there exists a Borel function 
$\rho_k$ which is admissible for $\Gamma_R$ and satisfies
$$
\int \rho_k^p \, d\mu \leq C'' < \infty, 
$$
where $C''$ does not depend on $k$. Moreover, by construction of $r_k$, the supports of $\rho_k$ are pairwise disjoint. We define $\rho^{\ell}$ by 
$$
\rho^{\ell}= \ell^{-1}\sum_{k=1}^{\ell} \rho_k. 
$$
Then $\rho^{\ell}$ is admissible for $\Gamma_R$, and 
$$
M_p(\Gamma_R) \leq \int (\rho^{\ell})^p \, d \mu \leq \ell^{1-p}C''  
$$
by disjointness of the supports. Therefore, since $p>1$, $M_p(\Gamma_R) \to 0$ as $\ell \to \infty$. The proof is complete. 
\end{proof}

We remark that Proposition~\ref{prop:intersect-path-modulus-hausdorff} also holds for $p=1$ under the stronger assumption that $\dim_{\haus}(A)<Q-1$. Indeed, in this case we may choose $p>1$ such that $\haus^{Q-p}(A)=0$. Let $R>0$ and $\rho^\ell$ be as in the proof above and note that $\rho^\ell$ is supported in $N(A, 2R)$. Thus H\"older's inequality applied to $\rho^\ell$ yields $$\|\rho^\ell\|_1\leq \|\rho^\ell\|_p \,\mu(N(A, 2R))^{\frac{p}{p-1}}\to 0\quad\text{as $\ell\to\infty$.}$$
This shows that $M_1(\Gamma_R)=0$ for every $R>0$ and thus $M_1(\Gamma_A)=0$.

\begin{proposition}\label{prop:capacity-intcurr-vs-modulus-curves}
 Let $X$ be a complete metric space, $\mu$ a Borel measure on $X$, and $m\geq 0$, $p\geq 1$. Let $A\subset X$ be a Borel set and $\Lambda\subset\AKic_m^0(X)$ a family satisfying $\spt T \subset A$ for every $T\in\Lambda$. If the family of (nonconstant) rectifiable paths with end-points in $A$ has zero $p$-modulus (in the usual sense) then $$\capa_p(\Lambda, \AKic_{m+1}(X))=0.$$
 \end{proposition}

\begin{proof}
 Denote by $\Gamma_A$ the family of nonconstant rectifiable paths with end-points in $A$. Since $M_p(\Gamma_A)=0$ there exists a Borel function $f\in L^p(X,\mu)$ with $f\geq 0$ and such that $$\int_\gamma f = \infty$$ for every $\gamma\in\Gamma_A$. Now, let $T\in \Lambda$ with $T\not=0$ and let $S\in\AKic_{m+1}(X)$ with $\partial S = T$. Suppose first that $m=0$. Let $c_i$ be Lipschitz curves as in Lemma~\ref{lem:weak-structure-1-currents}  for $S$, and
 denote by $c:[0,a]\to X$ the arc-length parameterization of $c_1$.  It follows from Lemmas~\ref{lemma:Lip-curve-current-mass-integral}  and \ref{lem:weak-structure-1-currents} that 
 \begin{equation*}
  \int_X f(x)d\|S\|(x) \geq \int_Xf(x)d\|c_{\#}\Lbrack 1_{[0,a]}\Rbrack\|(x) = \int_c f = \infty.
 \end{equation*}
Since $S$ was arbitrary it follows that $\capa_p(\Lambda, \AKic_{m+1}(X))=0$ in the case $m=0$. Now, suppose that $m\geq 1$.
Since $T\not=0$ there exists a Lipschitz map $\pi:X\to\R^m$ such that $T\rstr (1,\pi)\not=0$. We may assume that each component of $\pi$ is $1$-Lipschitz. By \cite[Theorem 5.6]{Ambrosio-Kirchheim-currents} we have
 \begin{equation*}
  \mass(T\rstr (1,\pi)) = \int_{\R^m} \mass(\langle T, \pi, y\rangle) dy,
 \end{equation*}
 where $\langle T, \pi, y\rangle$ denotes the slice of $T$ with respect to the map $\pi$ at $y$, see \cite{Ambrosio-Kirchheim-currents}.
 Thus there exists a measurable set $K\subset\R^m$ of strictly positive measure such that $\langle T, \pi, y\rangle\not=0$ for every $y\in K$. By \cite[Theorem 5.7]{Ambrosio-Kirchheim-currents} we may assume that $\langle S, \pi, y\rangle\in\AKic_1(X)$ and that $\bdry \langle S, \pi, y\rangle = (-1)^m \langle T, \pi, y\rangle$ is supported in $A\cap \pi^{-1}(\{y\})$ for every $y\in K$. Fix $y\in K$ and let $c_i$ be Lipschitz curves as in Lemma~\ref{lem:weak-structure-1-currents}  for $\langle S, \pi, y\rangle$. Let $c:[0,a]\to X$ be the arc-length parameterization of $c_1$.  Lemmas~\ref{lemma:Lip-curve-current-mass-integral}  and \ref{lem:weak-structure-1-currents} give 

 \begin{equation*}
  \int_X f(x)d\|\langle S, \pi, y\rangle\|(x) \geq \int_Xf(x)d\|c_{\#}\Lbrack 1_{[0,a]}\Rbrack\|(x) = \int_c f = \infty
 \end{equation*}
 for every $y\in K$, and since $K$ has strictly positive measure, it follows that
 \begin{equation*}
   \int_X fd\|S\| \geq \int_Xfd\|S\rstr d\pi\|
    \geq \int_{\R^m}\int_Xf(x)d\|\langle S, \pi, y\rangle\|(x)dy
    =\infty.
 \end{equation*}
 Since $S$ was arbitrary it follows that $\capa_p(\Lambda, \AKic_{m+1}(X))=0$.
\end{proof}


\section{Cochains in Lie groups} \label{part2}

\subsection{Statement of the main H\"older continuity estimates for cochains} \label{part21}

Let $G$ be a Lie group of dimension $n$, endowed with a left-invariant Finsler metric and the Hausdorff $n$-measure. Let $0\leq m\leq n-1$ and let $\mathcal{C} = (\mathcal{C}_m, \mathcal{C}_{m+1})$ with either $\mathcal{C}_k = \AKnc_k(G)$ or $\mathcal{C}_k=\AKic_k(G)$ for $k=m, m+1$. One of the principal aims of this paper is to give H\"older type estimates for $|\omega(T)|$ in terms of the flat norm or the filling volume of $T\in\mathcal{C}_m$. For this we will have to impose, in Theorems~\ref{hoeldercont2} and \ref{hoeldercont1} below, growth conditions on $T$ of the form
\begin{equation}\label{eq:growth-mass-T}
  \int_G\|T\|(B(z,r))^{\frac{1}{p-1}}d\|T\|(z)\leq A^{\frac{1}{p-1}} r^{\frac{\alpha}{p-1}}\mass(T)
  \end{equation}
  and
  \begin{equation}\label{eq:growth-mass-bdry-T}
\int_G\|\bdry T\|(B(z,r))^{\frac{1}{q-1}}d\|\bdry T\|(z)\leq B^{\frac{1}{q-1}} r^{\frac{\beta}{q-1}}\mass(\bdry T)
 \end{equation}
for suitable $p,q>1$, $\alpha,\beta>0$, and $r\geq 0$ and for some $A,B>0$. We remark that if
\begin{equation*}
 \|T\|(B(z,r))\leq A r^\alpha \quad\text{ and }\quad \|\bdry T\|(B(z,r))\leq B r^\beta
\end{equation*}
for every $z\in G$ then $T$ satisfies \eqref{eq:growth-mass-T} and \eqref{eq:growth-mass-bdry-T} for any $p,q>1$. Note that, by the convention established in Section~\ref{sec:metriccurrents}, inequality \eqref{eq:growth-mass-bdry-T} is an empty condition if $m=0$. We remark furthermore that there exist easy examples of currents $T$ which do not satisfy \eqref{eq:growth-mass-T} for any $\alpha$ and $p>1$. Indeed, let $n>m\geq 1$ and define $T\in\AKic^0_m(\R^n)$ by 
$$
T=\sum_{j=1}^\infty Q_j \varphi_{\#}(\bdry\Lbrack 1_{B_j}\Rbrack),  
$$
where $Q_j$ is the largest integer smaller than $2^{jm}j^{-2}$,  the map $\varphi: \R^{m+1}\to \R^n$ is given by $\varphi(x_1,\dots,x_{m+1}):= (x_1,\dots, x_{m+1}, 0,\dots,0)$, and $B_j\subset\R^{m+1}$ denotes the ball of radius $2^{-j}$ centered at $0$. Clearly, inequality  \eqref{eq:growth-mass-T} does not hold for any exponents $\alpha$ and $p>1$.

The main results of our paper can be stated as follows. In our first result we assume $\bdry T=0$.
 
\begin{theorem}\label{hoeldercont2}
Let $G$ and $m$ be as above and let $T\in\AKic_m^0(G)$. Suppose that either $G$ is a normed space or $\fillvol(T)\leq 1$. Suppose furthermore that there exist $A\geq 1$, $\alpha\in[0,m]$, $p> n-\alpha$ such that $T$ satisfies \eqref{eq:growth-mass-T} for all $r\geq 0$. If $\omega\in\mathcal{W}_p(\AKic_m^0(G))$ and $\omega$ has upper gradient $g$ then we have
\begin{equation}\label{eq:fillvol-hoelder-bound-G}
|\omega(T)|\leq E\Lambda(T)\, \fillvol(T)^{1-n/(p+\alpha)} \norm{g}_{p, N(\spt T, s_0)},
\end{equation}
with $$\Lambda(T) = \left(1 + \frac{p}{p+\alpha - n}\right)\left[A\mass(T)^{(p-1)}\right]^{n/(p(p+\alpha))},$$ and
$$s_0= E\cdot \fillvol(T)^{1/(m+1)}\left[1 + A^{-1/(p+\alpha)}\mass(T)^{\theta}\right],$$
where $E$ depends only on $m$, $G$, and the left-invariant Finsler metric on $G$, and where, moreover, $\theta=(1-\alpha/m)/(p+\alpha)$ if $m \geq 1$ and $\theta=(1-p)/p$ if $m=0$.  
\end{theorem}

Note that $N(\spt T, s_0)$ is defined at the beginning of Section~\ref{sec:notation}. The main principle behind Theorem \ref{hoeldercont2} is the following: the existence of a $p$-integrable upper gradient should force $\omega$ to be continuous with respect to the filling distance when $p$ is large enough, the same way as a Sobolev function with $p$-integrable gradient has to be continuous when $p>n$. However, in order for this principle to work we have to restrict $\omega$ to currents with controlled local growth, and the statement is therefore a bit technical. We give a simple corollary of Theorem \ref{hoeldercont2} to illustrate. Let $\omega \in \mathcal{W}_p(\AKic_m^0(\R^n))$ be as in the theorem. The theorem then implies that if $p>n-m$ then the restriction of $\omega$ to the class $\mathcal{S}_m$ of oriented $m$-dimensional spheres is continuous with respect to the filling distance; if 
$\fillvol(S_j -S)  \to 0$ for $S_j$, $S \in \mathcal{S}_m$, then $\omega(S_j -S) \to 0$. 

Theorem~\ref{hoeldercont2} together with Proposition~\ref{prop:upper-grad-lipfct-0-cochain} implies the following version of the Morrey-Sobolev inequality for Sobolev functions.
\begin{corollary}\label{cor:Morrey-Sobolev}
 Let $G$ be a Lie group of dimension $n$, endowed with a left-invariant Finsler metric and the Hausdorff $n$-measure. Let $u:G\to\overline{\R}$ be a function which has an upper gradient $g\in L^p(G)$ for some $p>n$. Then for all $x,y\in G$ with $d(x,y)\leq 1$ we have 
\begin{equation*}
 |u(x) - u(y)|\leq C d(x,y)^{1-n/p}\norm{g}_{p, B(x, Cd(x,y))},
\end{equation*}
where $C$ only depends on $p$, $G$, and the left-invariant Finsler metric.
\end{corollary}

Our second main result provides an analog of Theorem~\ref{hoeldercont2} for currents possibly with boundary.

\begin{theorem} \label{hoeldercont1}
Let $G$, $m$, and $\mathcal{C}$ be as above and let $T\in\mathcal{C}_m$. Suppose there exist $A,B>0$, $\alpha, \beta\in[0,n]$, $p> \max\{1, n-\alpha\}$, and $q>\max\{1, n-\beta\}$ such that $T$ satisfies \eqref{eq:growth-mass-T} and \eqref{eq:growth-mass-bdry-T} for all $r\geq 0$. If $\omega\in\mathcal{W}_{q,p}(\mathcal{C}_m)$ then we have
\begin{equation}\label{eq:flatnorm-hoelder-estimate}
|\omega(T)| \leq E\, \Lambda(T)\, \Big(\flatnormFAT(T, \mathcal{C})^{1-\frac{\lambda}{1+\delta}} + \flatnormFAT(T, \mathcal{C})^{1 - \frac{\gamma+\delta}{1+\delta}}  \Big)\,\|\omega\|_{q,p}
\end{equation}
with
$$
\Lambda(T) = \frac{2 + A^{1/p}+ B^{1/q}}{1-\gamma}\left(1 + \mass(T) + \mass(\bdry T)\right)^{\frac{\gamma + \delta}{1+\delta}},
$$
$$\delta= \max\{\alpha/p, \beta/q\}\quad\text{and}\quad \lambda = \min\{n/p, n/q\}\quad\text{and}\quad \gamma= \max\left\{\frac{n-\alpha}{p}, \frac{n-\beta}{q}\right\},$$
where $E$ depends only on $m$, $G$, and the left-invariant Finsler metric on $G$. 
\end{theorem}

Denote by $\mathcal{P}_m(\R^n)$ the space of real polyhedral $m$-chains in $\R^n$ and by $\mathcal{F}_m(\R^n)$ the completion with respect to the flat norm of $\mathcal{P}_m(\R^n)$. As a first consequence of Theorem \ref{hoeldercont1} we obtain that if $\omega$ is linear and belongs to 
$W_{q,p}(\AKnc_m(\R^n),\AKnc_{m+1}(\R^n))$, and if $p>n-m$ and $q> n-m+1$, then $\omega$ is well-defined for every $T \in \mathcal{P}_m(\R^n)$, in the sense that there exists a unique cochain $\omega' : \mathcal{P}_m(\R^n) \to \R$ such that the restriction of every representative of $\omega$ to $\mathcal{P}_m(\R^n)$ coincides with $\omega'$. However, unlike in the case of Whitney flat forms mentioned in the introduction and also below, $\omega'$ does not necessarily have a unique extension to the completion $\mathcal{F}_m(\R^n)$ because our estimates depend on the local mass growths of $T$ and $\partial T$. 


As a second consequence of Theorem~\ref{hoeldercont1} we obtain the following statement about Whitney flat forms and thus partly recover Wolfe's theorem mentioned in the introduction. Every flat $m$-form $\omega$ in $\R^n$ gives rise to a unique linear cochain $\tilde{\omega}:\mathcal{F}_m(\R^n)\to\R$ which is Lipschitz continuous with respect to the flat norm; more precisely,  
\begin{equation}\label{eq:lip-cont-omega-flat}
 |\tilde{\omega}(T)| \leq \tilde{E}\, \flatnormFAT(T) \|\omega\|_\flat
\end{equation}
for every $T\in\mathcal{F}_m(\R^n)$, where $\tilde{E}$ is independent of $T$ and $\omega$. The assignment $\omega\mapsto\tilde{\omega}$ is linear and injective, and $\|\omega\|_\flat$ is defined by $\|\omega\|_\flat = \max\{\|\omega\|_\infty, \|d\omega\|_\infty\}.$
Note that Wolfe's theorem asserts the same with $\tilde{E}=1$; moreover, it provides a converse.
We briefly sketch how Theorem~\ref{hoeldercont1} implies the statement above. Let $\omega$ first be a flat $m$-form in $\R^n$ with compact support. By the discussion after Definition~\ref{sobospace} the form $\omega$ gives rise to a linear cochain $\tilde{\omega}^k$ in $W_{k,k}(\AKnc_m(\R^n),\AKnc_{m+1}(\R^n))$ for every $k\in\N$. It follows from the paragraph above that for all $k$ large enough $\tilde{\omega}^k(T)$ is well-defined for every $T\in\mathcal{P}_m(\R^n)$; moreover, for every $k$ large enough we have $\|\tilde{\omega}^k\|_{k,k}\leq C\|\omega\|_\flat$ for some constant $C$ which is  independent of $k$. Finally, Lemma~\ref{fugledelemma} together with Proposition~\ref{elementary} show that for every $T\in\mathcal{P}_m(\R^n)$ we have $\tilde{\omega}^k(T) = \tilde{\omega}^l(T)$ for all $k,l$ large enough. We can therefore define a linear cochain $\tilde{\omega}$ on $\mathcal{P}_m(\R^n)$ by $\tilde{\omega}(T):= \lim_{k\to\infty}\tilde{\omega}^k(T)$. Since in Theorem~\ref{hoeldercont1}, the exponents of $\flatnormFAT$ tend to $1$ and $\Lambda(T)\to2$ when $p,q\to\infty$ it follows that $\tilde{\omega}$ indeed satisfies \eqref{eq:lip-cont-omega-flat}, and clearly, $\tilde{\omega}$ extends to $\mathcal{F}_m(\R^n)$. 
To prove the assertion for general, not necessarily compactly supported, flat forms $\omega$, we fix $T \in \mathcal{P}_m(\R^n)$, and a ball $B(0,R)$ containing the support of $T$. Moreover, we choose a smooth compactly supported function $\varphi_j$ with the following properties: $\varphi_j$ takes values between $0$ and $1$, equals $1$ on $B(0,R)$, and $|\nabla \varphi_j|$ is bounded by $1/j$. Next, we define $\omega_j$ by multiplying the coefficients of $\omega$ by $\varphi_j$. We can now define 
$\tilde{\omega}(T)=\tilde{\omega}_j(T)$ as above; the definition is clearly independent of $j$. Moreover, 
$|\omega_j|_{\flat} \to |\omega|_{\flat}$, so we can apply the above argument with the compactly supported forms $\omega_j$ to get the conclusion also for $\omega$. 

Finally, we note that Theorems \ref{hoeldercont2} and \ref{hoeldercont1} do not in general hold for the borderline exponents $p=n-\alpha$ and $q=n-\beta$ as the following example shows.

\begin{example}
\label{sharpthm}
Let $n \geq 3$ and $1 \leq m \leq n-2$. Moreover, let $T=\Lbrack 1_{B(0,1)}\Rbrack\in\AKic_m(\R^m)$ and define $F:\R^m \to \R^n$ by
$$
F(y_1,\ldots,y_m)=(0,\ldots,0,y_{1}, \ldots, y_m).
$$ 
For $x \in \R^n$ let $T_x:= \psi_{x\#} T$, where the map $\psi_x$ is given by $\psi_x(y):=F(y)+x$. Then the currents $T_x$ satisfy \eqref{eq:growth-mass-T} with $\alpha = m$ and \eqref{eq:growth-mass-bdry-T} with $\beta=m-1$. Fix a smooth $\varphi: \R^m \to [0,1]$ such that 
$\varphi$ equals $1$ on $B(0,1)$ and $0$ on $\R^m \setminus B(0,2)$. Finally, denote $\tilde{x}=(x_1,\ldots, x_{n-m})$, and define an $m$-form $\tilde{\omega}$ on $\R^n$ by
$$
\tilde{\omega}(x) = \varphi(x_{n-m+1},\ldots,x_n)\max \{\log (\log 1/|\tilde{x}|),0 \}\,  dx_{n-m+1} \wedge \cdots \wedge dx_n. 
$$
We have 
$$
|\tilde{\omega}| \leq 1_{\R^{n-m} \times B(0,2)}\max\{\log (\log 1/|\tilde{x}|),0\} \in L^q(\R^n) \quad \text{for every } q \geq 1, 
$$
and 
$$
| d \tilde{\omega}| \leq 1_{\R^{n-m} \times B(0,2)} \max\{|\tilde{x}|^{-1}(\log 1/|\tilde{x}|)^{-1},1/e\} \in L^{n-m}(\R^n).   
$$
By the discussion in Section \ref{exception}, $\tilde{\omega}$ induces a cochain $\omega \in \mathcal{W}_{q,n-m}(\AKic_m(\R^n))$ for every $q \geq 1$. However, $\omega(T_x)$ converges to infinity as $x \to 0$. This shows that Theorem \ref{hoeldercont1} does not hold with the borderline exponents. By slightly modifying the example, we see that this is the case also for Theorem \ref{hoeldercont2}; instead of an $m$-ball, let $T$ be induced by an $m$-sphere. Then we can construct a cochain $\omega \in \mathcal{W}_{n-m}(\AKic_m^0(\R^n))$ with a singularity at $T$ in a similar way as above. 
\end{example}

The proofs of the two theorems above will be given at the end of Section~\ref{section:tech-est-cochains}. In Section~\ref{part22} and most of \ref{section:tech-est-cochains} we prove auxiliary results used in the proofs of the two theorems.  We briefly discuss the main geometric ideas of the proof of Theorem \ref{hoeldercont2}. Let $T\in\AKic_m^0(G)$  and $\omega\in\mathcal{W}_p(\AKic_m^0(G))$. Using the group structure of $G$, we show that the averages 
$$
 \omega_+(T, r):=  \frac{1}{\haus^n(B(e,r))}\int_{B(e,r)} |\omega(\varphi_{x\#}T)|\,d\haus^n(x) 
$$
are well-defined, where $\varphi_x$ is the right-multiplication by $x$. The proof of the theorem is based on a simple change of variables formula (Lemma \ref{lemma:fubini-G}), and two basic estimates concerning  $\omega_+(T, r)$. Firstly, we take almost minimal fillings of the currents $\varphi_{x\#}T$, and then estimate $\omega_+(T, r)$ using the upper gradient property of $\omega$ over the fillings, and change of variables. We also use isoperimetric methods to show that we can restrict ourselves to a small neighborhood of the support of $T$. Secondly, we fill $T-\varphi_{x\#}T$ with a current whose geometry is suitably controlled, using a notion of controlled family of curves. In euclidean space we could simply choose this family of curves to be geodesic segments transporting $T$ to $\varphi_{x\#}T$. We then estimate the difference $| \omega(T)- \omega_+(T, r)|$, using the upper gradient property of $\omega$ over these fillings, and change of variables. In this second step we need to be able to control the local growth of the fillings, and it is for this reason that we need to assume local growth conditions on $T$. Finally, we combine the two estimates and choose the radius $r$ in an optimal way to finish the proof. The proof of Theorem \ref{hoeldercont1} follows the same steps but estimates concerning the boundary of $T$ also come into play.


\subsection{Basic integral estimates} \label{part22}

The aim of this as well as most of the next section is to develop the tools which will allow us to prove the H\"older continuity estimates stated in the previous section. 

Let $G$ be a Lie group of dimension $n$, endowed with a left-invariant Finsler metric and the Hausdorff $n$-measure. We first prove the following estimate.

\begin{proposition}\label{proposition:elem-upper-G}
 Given a finite Borel measure $\mu$ on $G$, a Borel measurable function $f:G\to[0,\infty]$, a Borel set $A\subset G$, and $p\geq1$, we have
\begin{equation}\label{eq:int-est-upper-G}
 \int_A\int_Gf(zx)d\mu(z)d\haus^n(x) \leq \|f\|_{p, \Omega} \left[\haus^n(A)\mu(G)\right]^{\frac{p-1}{p}} \varrho(\mu, A^{-1})^{\frac{1}{p}},
\end{equation} 
where we have set $\Omega = (\spt\mu)\cdot A$ and $$\varrho(\mu, C):= \sup_{z\in G}\mu(zC)$$ whenever $C\subset G$ is Borel measurable.
\end{proposition}


We first note:

\begin{lemma}\label{lemma:fubini-G}
Given a Borel measure $\mu$ on $G$, a Borel measurable function $f:G\to[0,\infty]$, and a Borel set $A\subset G$, we have
\begin{equation*}
 \int_A\int_Gf(zx)d\mu(z)d\haus^n(x) = \int_Gf(x)\mu(xA^{-1})d\haus^n(x).
\end{equation*} 
\end{lemma}

\begin{proof}
 We have $$\mu(xA^{-1}) = \int_G 1_{xA^{-1}}(z)d\mu(z) = \int_G1_A(z^{-1}x)d\mu(z)$$ and hence, by Fubini-Tonelli theorem and left-invariance of $\haus^n$,
 \begin{equation*}
  \begin{split}
   \int_Gf(x)\mu(xA^{-1})d\haus^n(x) &= \int_G\int_Gf(x)1_A(z^{-1}x)d\mu(z)d\haus^n(x)\\
    &= \int_G\int_Gf(zx)1_A(x)d\haus^n(x)d\mu(z)\\
    &= \int_A\int_Gf(zx)d\mu(z)d\haus^n(x).
  \end{split}
 \end{equation*}
\end{proof}

\begin{proof}[Proof of Proposition~\ref{proposition:elem-upper-G}]
 Note first that if $x\not\in (\spt \mu)\cdot A$ then $xA^{-1}\cap \spt \mu=\emptyset$ and thus $\mu(xA^{-1})=0$. In case $p=1$ then inequality \eqref{eq:int-est-upper-G} follows directly from Lemma~\ref{lemma:fubini-G}. If $p>1$ then we use Lemma~\ref{lemma:fubini-G}, the Fubini-Tonelli theorem, and 
 H\"older's inequality, to obtain
 \begin{equation*}
  \begin{split}
   \int_A\int_Gf(zx)d\mu(z)d\haus^n(x) &= \int_Gf(x)\mu(xA^{-1})d\haus^n(x)\\
   &\leq \|f\|_{p, \Omega}\left(\int_G\mu(xA^{-1})^{\frac{1}{p-1}}\mu(xA^{-1})d\haus^n(x)\right)^{\frac{p-1}{p}}\\
   &= \|f\|_{p,\Omega}\left(\int_A\int_G\mu(zxA^{-1})^{\frac{1}{p-1}}d\mu(z)d\haus^n(x)\right)^{\frac{p-1}{p}}\\
   &\leq \|f\|_{p, \Omega} \left[\haus^n(A)\mu(G)\right]^{\frac{p-1}{p}} \varrho(\mu, A^{-1})^{\frac{1}{p}}.
  \end{split}
 \end{equation*}
\end{proof}

By using the fact that $zxA^{-1}\subset zAA^{-1}$ for every $x\in A$ in the proof above we also obtain the following variant of Proposition~\ref{proposition:elem-upper-G}.

\begin{proposition}\label{proposition:elem-upper-G-variant}
 Given a finite Borel measure $\mu$ on $G$, a Borel measurable function $f:G\to[0,\infty]$, a Borel set $A\subset G$, and $p>1$, we have
\begin{equation*}
 \int_A\int_Gf(zx)d\mu(z)d\haus^n(x) \leq \|f\|_{p, \Omega} \haus^n(A)^{\frac{p-1}{p}} \left(\int_G \mu(zAA^{-1})^{\frac{1}{p-1}}d\mu(z)\right)^{\frac{p-1}{p}},
\end{equation*} 
where $\Omega = (\spt\mu)\cdot A$.
\end{proposition}

\begin{remark}
Note for example that if $A=B(e,r)$ then $zAA^{-1} \subset B(z,2r)$.
\end{remark}

We can use Proposition~\ref{proposition:elem-upper-G} to obtain the following estimate in Euclidean space.

\begin{proposition}\label{prop:growth-integration-Euclidean}
 Let $n\geq 1$ and let $\mu$ be a finite Borel measure on $\R^n$ such that, for some $A, r_0>0$ and $\alpha\in[0, n]$,
 \begin{equation}\label{eq:strong-growth-condition-mu}
  \mu(B(z,r))\leq Ar^\alpha
 \end{equation}
 for all $r\in(0,r_0)$ and all $z\in \R^n$. Let $f:\R^n\to[0,\infty]$ be Borel measurable and $p>\max\{1, n-\alpha\}$. Then for every $r\in(0,r_0)$ we have
 \begin{equation*}
  \int_{B(0,r)}\int_0^1\int_{\R^n} f(z+tx)\,d\mu(z)dt\,dx \leq \frac{A^{\frac{1}{p}}p}{p-n+\alpha}r^{\frac{\alpha}{p}} \left[\omega_nr^n\mu(\R^n)\right]^{\frac{p-1}{p}} \|f\|_{p, \Omega},
 \end{equation*}
 where $\Omega = N(\spt\mu, r)$ and where $\omega_n$ denotes the volume of the unit ball in $\R^n$. 
\end{proposition}

\begin{proof}
 We simply use the change of variable formula and Proposition~\ref{proposition:elem-upper-G} to calculate
  \begin{equation*}
   \begin{split}
     \int_{B(0,r)}\int_0^1\int_{\R^n} f(z+tx)\,d\mu(z)dt\,dx &= \int_0^1t^{-n} \int_{B(0,tr)}\int_{\R^n} f(z+x)\,d\mu(z)\,dx\,dt\\
     &\leq A^{\frac{1}{p}}r^{\frac{\alpha}{p}}  \left[\omega_nr^n\mu(\R^n)\right]^{\frac{p-1}{p}} \|f\|_{p, \Omega} \int_0^1 t^{\frac{-np + n(p-1) + \alpha}{p}}dt\\
     &= \frac{A^{\frac{1}{p}}p}{p-n+\alpha}r^{\frac{\alpha}{p}} \left[\omega_nr^n\mu(\R^n)\right]^{\frac{p-1}{p}} \|f\|_{p, \Omega}.
   \end{split}
  \end{equation*}
\end{proof}

We now generalize Proposition~\ref{prop:growth-integration-Euclidean} to the setting of Lie groups. For this, we first make the following technical definition. 

\begin{defn}\label{def:controlled-family-curves}
 Let $M$ be a manifold with distance $d$ coming from a Finsler metric, and let $x_0\in M$ and $r_0>0$. We say that the ball $B(x_0,r_0)$ admits a $(C,s, \lambda, \eta)$-controlled family of curves, where $C, \lambda \geq 1$ and $s, \eta>0$,  if there exists a Lipschitz map $$H: [0,1]\times B(x_0,r_0)\to (M,d)$$ such that for all $x, t$ we have $H(0,x)=H(t,x_0)=x_0$ and $H(1,x) = x$, and furthermore, the curve $t\mapsto H(t,x)$ is $\eta$-Lipschitz for every $x$; finally, for every $t\in(0,1]$ the map $H_t(x):= H(t,x)$ is injective, satisfies $H_t(B(x_0,r))\subset B(x_0, \lambda tr)$ for all $r\in(0,r_0)$, and the jacobian of $H_t$ is bounded by
 \begin{equation*}
  C^{-1} t^s \leq J_n(d_x H_t) \leq C t^s
 \end{equation*}
 for almost every $x\in B(x_0,r_0)$.
 \end{defn}

Note that if $B(x_0,r_0)$ admits a $(C,s, \lambda, \eta)$-controlled family of curves then so does $B(x_0, r)$ for every $r\in(0,r_0)$; indeed the restriction of $H$ to $[0,1]\times B(x_0,r)$ clearly defines a $(C,s, \lambda, \eta)$-controlled family of curves.
We now give several examples of  manifolds with controlled families of curves. 
\begin{enumerate}
\item Let $M$ be a manifold of dimension $n$ with a Finsler metric, and let $x_0\in M$ and $r_0>0$ be such that there exists a biLipschitz map $F: B_E(R) \to M$, where $B_E(R)$ is the Euclidean $n$-ball of radius $R$ centered at $0$, such that $F(0) = x_0$ and such that $B(x_0, r_0)\subset F(B_E(R))$. Then $B(x_0, r_0)$ admits a $(C^n, n, C, Cr_0)$-controlled family of curves, where $C$ only depends on the biLipschitz constant of $F$. Indeed, if $F$ is a $\bar{C}$-biLipschitz map as above then the map $H(t,x):= F(t\cdot F^{-1}(x))$ satisfies all the properties with $C=\bar{C}^2$. Note that, if $G$ is a Carnot group of topological dimension $n$ then the Lie exponential map is a global diffeomorphism, and thus for every $r_0$ there exists $C$ such that every ball $B(x_0, r)$ with $r\leq r_0$ admits a $(C^n, n, C, Cr)$-controlled family of curves. Likewise, if $M$ is a Riemannian manifold of dimension $n$, $x_0\in M$ and $0< r_0<\operatorname{injrad}_{x_0}(M)$ then the exponential map $\exp_{x_0}: B(0,r_0)\subset T_{x_0}M\to B(x_0, r_0)$ is a diffeomorphism, and hence $B(x_0, r)$ admits a $(C^n, n, C, Cr)$-controlled family of curves for every $0<r<r_0$, where $C$ is a constant.
\item Let $G$ be a Carnot group of step $c$ and homogeneous dimension $Q$, endowed with a left-invariant Finsler metric. Then there exists a constant $D$ such that every ball $B(x_0,r_0)$ in $G$ admits a $(1, Q, 1, D\tau(r_0))$-controlled family of curves, where  
  \begin{equation}\label{eq:def-tau-r}
     \tau(r):=\left\{\begin{array}{ll} 
   r & 0\leq r\leq 1\\
   r^c & 1< r.  \end{array}\right.
   \end{equation}
Indeed, one can prove that the map $H(t,x):= \delta_t(x)$, where $\delta_t$ is the dilatation homomorphism, satisfies all the desired properties for $x_0=e$, where $e$ denotes the identity element in $G$. Since left-translations are isometries the result follows. The only non-trivial part in the above is to prove the estimate on the Lipschitz constant. This is done in the lemma below.
\end{enumerate}

\begin{lemma}\label{lemma:Lip-distortions}
 Let $G$ be a Carnot group of step $c$, endowed with a left-invariant Finsler metric $d_0$. Then there exists a constant $D$ such that for all $x\in G$ the curve $\gamma:[0,1]\to G$ given by $\gamma(t):= \delta_t(x)$ is $D\tau(|x|)$-Lipschitz, where we have abbreviated $|x|:= d_0(e,x)$. 
\end{lemma}

\begin{proof}
 Let $\mathfrak{g} = V_1\oplus\dots\oplus V_c$ be a stratification of the Lie algebra $\mathfrak{g}$ of $G$.  Endow $\mathfrak{g}$ with an inner product such that the $V_j$ are pairwise orthogonal.  Let $\bar{\delta}_t: \mathfrak{g}\to\mathfrak{g}$ denote the Lie algebra homomorphism such that $\bar{\delta}_t(v) = t^j v$ for every $v\in V_j$ and every $j=1,\dots, c$ . Then the dilatation homomorphism $\delta_t$ satisfies $\delta_t \circ\exp = \exp\circ\bar{\delta}_t$, where $\exp\mathfrak{g}\to G$ is the  Lie exponential map. Note that $\exp$  is a diffeomorphism and, in particular, a local biLipschitz homeomorphism. Let $R>0$ be large enough so that $\exp(B(0,R))$ contains the unit ball around the identity in $G$. Let $C$ be the biLipschitz constant of $\exp|_{B(0,R)}$. Let $v\in B(0,R)$ be such that $x=\exp(v)$. It is straightforward to check that the map $t\mapsto \bar{\delta}_t(v)$ is $|v|$-Lipschitz on $[0,1]$. It thus follows that 
 $$d(\delta_t(x), \delta_s(x)) = d(\exp(\bar{\delta}_t(v)), \exp(\bar{\delta}_s(v)))\leq C |\bar{\delta}_t(v) - \bar{\delta}_s(v)| \leq C^2|t-s|\,|x|$$ and hence the claim (with $D= C^2$) in the case that $|x|\leq 1$. Now, suppose that $|x|>1$. Define $r:= |x|^{-1}$. We then have that $|\delta_r(x)|\leq r|x| = 1$ and $$\gamma(t) = \delta_{\frac{1}{r}}\circ \delta_t(\delta_r(x)).$$ Since $\delta_{\frac{1}{r}}$ is $(1/r)^c$-Lipschitz it thus follows with the above that $\gamma$ is $Dr^{-c}$-Lipschitz, as claimed.
\end{proof}

We now prove the following generalization of Proposition~\ref{prop:growth-integration-Euclidean}.

\begin{proposition}\label{proposition:growth-integration-variant}
 Let $G$ be a Lie group of dimension $n$, endowed with a left-invariant Finsler metric. Let $r_0>0$ and suppose $B(e, r_0)$ admits a $(C,s, \lambda, \eta)$-controlled family of curves, defined by a Lipschitz map $H$ as in Definition~\ref{def:controlled-family-curves}. Let $\mu$ be a finite Borel measure on $G$ such that, for some $A>0$, $\alpha\in[0, s]$, and $p>\max\{1, s-\alpha\}$,
 \begin{equation}\label{eq:weak-growth-condition-mu}
  \int_G\mu(B(z,r))^{\frac{1}{p-1}}d\mu(z)\leq A^{\frac{1}{p-1}} r^{\frac{\alpha}{p-1}}\mu(G)
 \end{equation}
 for all $r\in(0,2\lambda r_0)$. Then for every Borel measurable function $f:G\to[0,\infty]$ and every $r\in(0,r_0)$ we have
 \begin{equation*}
  \int_{B(e,r)}\int_0^1\int_G f(zH_t(x))d\mu(z)dt\,d\haus^n(x) \leq \bar{A}\, r^{\frac{\alpha}{p}} \left[\haus^n(B(e,r))\mu(G)\right]^{\frac{p-1}{p}} \|f\|_{p, \Omega},
 \end{equation*}
 where $$\bar{A} = \frac{C^2A^{\frac{1}{p}}(2\lambda)^{\frac{\alpha}{p}}p}{p-s+\alpha}$$ and $\Omega = (\spt\mu)\cdot B(e,\lambda r) = N(\spt\mu, \lambda r)$.
\end{proposition}

Note that \eqref{eq:weak-growth-condition-mu} is exactly \eqref{eq:growth-mass-T} when $\mu=\|T\|$. Note also that \eqref{eq:weak-growth-condition-mu} is for example satisfied if
\begin{equation*}
  \mu(B(z,r))\leq Ar^\alpha
 \end{equation*}
 for all $r\in(0,2\lambda r_0)$ and all $z\in G$. Note furthermore that the value of $\eta$ is of no importance in the above proposition and it does not appear in the estimate. It will only be of importance when we use the above proposition in Section \ref{section:tech-est-cochains}.

\begin{proof}
 Let $H:[0,1]\times B(e, r_0)\to G$ be the Lipschitz map defining the controlled family of curves. Fix $r\in(0,r_0)$ and note that the restriction of $H$ to $[0,1]\times B(e, r)$ defines a $(C,s, \lambda, \eta)$-controlled family of curves on $B(e,r)$. Note that $$\haus^n(H_t(B(e,r))) \leq C t^s \haus^n(B(e,r)).$$
 We now use the change of variable formula and Proposition~\ref{proposition:elem-upper-G-variant} in order to compute, with $B:= B(e,r)$, that
\begin{equation*}
  \begin{split}
    \int_{B(e,r)}\int_0^1&\int_G f(zH_t(x))d\mu(z)dt\,d\haus^n(x) \\
     &=  \int_0^1\int_{H_t(B)}\int_Gf(zx) J_n(d_{H_t^{-1}(x)}H_t)^{-1} d\mu(z)d\haus^n(x)dt\\
     &\leq C  \int_0^1t^{-s} \int_{H_t(B)}\int_Gf(zx)  d\mu(z)d\haus^n(x)dt\\
     &\leq C \|f\|_{p,\Omega} \int_0^1t^{-s}\haus^n(H_t(B))^{\frac{p-1}{p}}  \left(\int_G \mu(zH_t(B)H_t(B)^{-1})^{\frac{1}{p-1}}d\mu(z)\right)^{\frac{p-1}{p}}dt\\
     &\leq C^{2-\frac{1}{p}} \|f\|_{p,\Omega} \haus^n(B)^{\frac{p-1}{p}} \int_0^1t^{-\frac{s}{p}} \left(\int_G \mu(B(z,2\lambda tr))^{\frac{1}{p-1}}d\mu(z)\right)^{\frac{p-1}{p}}dt\\
     &\leq C^{2-\frac{1}{p}} A^{\frac{1}{p}}(2\lambda)^{\frac{\alpha}{p}}r^{\frac{\alpha}{p}} \left[\haus^n(B)\mu(G)\right]^{\frac{p-1}{p}} \|f\|_{p, \Omega}\int_0^1t^{\frac{\alpha-s}{p}}dt\\
     &= \frac{C^{2-\frac{1}{p}}A^{\frac{1}{p}}(2\lambda)^{\frac{\alpha}{p}}p}{p-s+\alpha}r^{\frac{\alpha}{p}} \left[\haus^n(B(e,r))\mu(G)\right]^{\frac{p-1}{p}} \|f\|_{p, \Omega}.
  \end{split}
 \end{equation*}
\end{proof}


\subsection{Technical estimates for cochains and the proof of H\"older continuity}\label{section:tech-est-cochains}

In this section we will use the results from the previous section in order to prove Theorems~\ref{hoeldercont2} and \ref{hoeldercont1}. 

Let $G$ be a Lie group, endowed with a left-invariant Finsler metric $d_0$. For $x\in G$ let $\varphi_x$ denote the right-multiplication map by $x$, that is, $\varphi_x(z):= zx$. Define a function $\bar{\tau}_G:[0,\infty)\to[0,\infty)$ by
\begin{equation*}
 \bar{\tau}_G(r):= \max\left\{\|\operatorname{Ad}_x \|: x\in B(e, r) \right\},
\end{equation*}
where $\operatorname{Ad}_x$ is the adjoint, that is, $\operatorname{Ad}_x = d_e\Psi_x$ with $\Psi_x(z):= xzx^{-1}$, and where $\|\cdot\|$ denotes the operator norm on $T_eG$. In the following we will write $\bar{\tau}(r)$ instead of $\bar{\tau}_G(r)$ if there is no risk of ambiguity. It is easy to check that $\varphi_x$ is $\bar{\tau}_G(|x|)$-Lipschitz, where $|x|:= d_0(e,x)$. In general, it seems difficult to determine an explicit upper bound for $\bar{\tau}_G(r)$, however, in the following case this is possible.

\begin{lemma}\label{lemma:Lip-right-mult-Carnot}
 Let $G$ be a Carnot group of step $c$, endowed with a left-invariant Finsler metric. Then there exists a constant $D$ such that 
         \begin{equation}\label{eq:def-tau-bar-r}
     \bar{\tau}_G(r)\leq \left\{\begin{array}{ll} 
   D & 0\leq r\leq 1\\
   Dr^{c-1} & 1< r.  \end{array}\right.
   \end{equation} 
\end{lemma}

\begin{proof}
 Denote the left-invariant Finsler metric by $d_0$. It is clear from the above that there exists $D$ such that $\bar{\tau}_G(r)\leq D$ for all $0<r\leq 1$. Now, let $x\in G$ with $d_0(x, e)>1$. Set $r:= d_0(x,e)^{-1}$ and note that $\delta_r$ is $r$-Lipschitz while $\delta_{1/r}$ is $r^{-c}$-Lipschitz. Since $d_0(\delta_r(x), e)\leq 1$ and $$\varphi_x = \delta_{1/r} \circ \varphi_{\delta_r(x)} \circ \delta_r$$ it follows immediately that $\varphi_x$ is $Dr^{-(c-1)}$-Lipschitz. Finally, since left-multiplication is an isometry we conclude that $$\|\operatorname{Ad}_x(v)\| \leq Dr^{-(c-1)} \|v\|$$ for every $v\in T_eG$ and hence the claim.
\end{proof}

\begin{lemma}\label{lemma:sobolev-omega-translates}
 Let $G$ be a Lie group of dimension $n$, endowed with a left-invariant Finsler metric. Let $0\leq m\leq n-1$ and $\mathcal{C} = (\mathcal{C}_m, \mathcal{C}_{m+1})$ with either $\mathcal{C}_k = \AKnc_k(G)$ or $\mathcal{C}_k=\AKic_k(G)$ for $k=m, m+1$. Suppose $\omega$ is a cochain on $\mathcal{C}$ and $T\in\mathcal{C}$. Then the function $u: G\to[0,\infty]$ defined by $u(x):= |\omega(\varphi_{x\#}T)|$ has the following properties:
 \begin{enumerate}
  \item[(i)] if $\omega\in\mathcal{W}_{q,p}(\mathcal{C})$ for some $1\leq p,q<\infty$ then $u\in W^{1,\kappa}_{\rm loc}(G)$ with $\kappa=\min\{p,q\}$;
    \item[(ii)] if $\omega\in\mathcal{W}_{p}(\mathcal{C})$ for some $1\leq p<\infty$ and if $\fillvol(T, \mathcal{C}_{m+1})<\infty$ then $u\in W^{1,p}_{\rm loc}(G)$.
   \end{enumerate}
 \end{lemma} 

In case $\omega$ is a linear cochain then statements (i) and (ii) also hold for $u(x):=\omega(\varphi_{x\#}T)$. In case $\fillvol(T, \mathcal{C}_{m+1})<\infty$  then it is in fact enough if $\omega$ is a cochain on $\mathcal{C}^0$. Note that in statement (ii) one cannot replace the condition $\fillvol(T, \mathcal{C}_{m+1})<\infty$ by $\bdry T = 0$ in general.

\begin{proof}
 We only prove statement (i) because the proof of statement (ii) is analogous.
 Let $h \in L^q(G)$ be an upper norm of $\omega$ and $g\in L^p(G)$ be an upper gradient of $\omega$. We first show that $u(x)$ is finite for almost every $x\in G$. For this suppose to the contrary that there exists a Borel set $B\subset G$ of strictly positive measure such that $u(x)=\infty$ for every $x\in B$. We may assume without loss of generality that $B$ is contained in the ball $B(e,r)$ for some $r<\infty$. Since $$\|\varphi_{x\#}T\|\leq \bar{\tau}(|x|)^m\varphi_{x\#}\|T\|\leq \bar{\tau}(r)^m\varphi_{x\#}\|T\|$$ for every $x\in B$ we obtain from Proposition~\ref{proposition:elem-upper-G} that 
\begin{equation*}
 \begin{split}
  \infty &= \int_Bu(x)d\haus^n(x)\\
   &\leq \bar{\tau}(r)^m \int_B\int_Gh(zx)d\|T\|(z)d\haus^n(x)\\
   & \leq \bar{\tau}(r)^m \|h\|_p\mass(T)[\haus^n(B)]^{\frac{p-1}{p}},
  \end{split}
 \end{equation*}
which gives a contradiction. This shows that $u(x)$ is indeed finite for almost every $x\in G$.
Now, define a function $v: G\to[0,\infty]$ by 
\begin{equation*}
 v(x):= (m+1)\bar{\tau}(|x|)^m \int_G g(zx) d\|T\|(z) + m\bar{\tau}(|x|)^{m-1}\int_G h(zx)d\|\partial T\|(z).
\end{equation*}
Then $v$ is Borel measurable and locally in $L^\kappa(G)$ since, by Jensen inequality and Lemma~\ref{lemma:fubini-G},
\begin{equation*}
 \begin{split}
 \int_{B(e, r)}&\left(\bar{\tau}(|x|)^m \int_G g(zx)d\|T\|(z)\right)^pd\haus^n(x)\\
 &\leq \bar{\tau}(r)^{pm} \mass(T)^{p-1}  \int_{B(e, r)}\int_G g(zx)^pd\|T\|(z)d\haus^n(x)\\
&\leq \bar{\tau}(r)^{pm} \mass(T)^p \int_Gg(x)^pd\haus^n(x)\\
 &= \bar{\tau}(r)^{pm} \mass(T)^p \|g\|_p^p <\infty,
\end{split}
\end{equation*}
and analogously, in the case $m\geq 1$,
\begin{equation*}
  \int_{B(e, r)}\left(\bar{\tau}(|x|)^{m-1} \int_G h(zx)d\|\bdry T\|(z)\right)^qd\haus^n(x) \leq \bar{\tau}(r)^{q(m-1)} \mass(\partial T)^q \|\omega\|_q^q<\infty
\end{equation*}
for every $r>0$.
Now, let $B\subset G$ be a Borel set with $\haus^n(B)=0$ and such that $u(x)<\infty$ for all $x\not\in B$. Define $\bar{v}$ by $\bar{v}(x) = v(x)$ if $x\not\in B$ and $\bar{v}(x) = \infty$ if $x\in B$. It follows that $\bar{v}$ is Borel measurable and locally in $L^\kappa(G)$. We show that $\bar{v}$ is an upper gradient of the function $u$. For this, let $a,b\in G$ and let $\gamma:[0,1]\to G$ be a rectifiable curve joining $a$ and $b$, parameterized proportional to arc-length. We must show that $$|u(b) - u(a)|\leq \int_0^1\bar{v}\circ\gamma(t) |\dot{\gamma}(t)|dt,$$ where it is understood that the right hand side must equal $\infty$ in case $u(a)=\infty$ or $u(b)=\infty$.
Define $\psi:[0,1]\times G\to G$ by $\psi(t,z):= z\gamma(t)$ and note that $\psi(t, \cdot)$ is $\bar{\tau}(|\gamma(t)|)$-Lipschitz for every $t\in[0,1]$ and $\psi(\cdot, z)$ is $\Lipconst(\gamma)$-Lipschitz for every $z\in G$. Define $S:= \psi_{\#}([0,1]\times T)$. If $m\geq 1$ define $R:= \psi_{\#}([0,1]\times \partial T)$, if $m=0$ then set $R=0$. If $T$ is a normal current then so are $S$ and $R$. If $T$ is an integral current then so are $S$ and $R$.
Clearly, we have $$\varphi_{b\#}T - \varphi_{a\#}T = \partial S + R.$$ It follows from Lemma~\ref{lemma:mass-estimate-push-cone} that $$\|S\| \leq (m+1)\Lipconst(\gamma)\; \psi_{\#}\left[\bar{\tau}(|\gamma(\cdot)|)^m\lm^1\times \|T\|\right]$$ and $$\|R\| \leq m\Lipconst(\gamma)\; \psi_{\#}\left[\bar{\tau}(|\gamma(\cdot)|)^{m-1} \lm^1\times \|\partial T\|\right].$$ If $u(a)<\infty$ or $u(b)<\infty$ we conclude that
\begin{equation*}
 \begin{split}
  |u(b) - u(a)| &= |\,|\omega(\varphi_{b\#}T)| - |\omega(\varphi_{a\#}T)|\,|\\
  &\leq |\omega(\partial S)| + |\omega(R)|\\
  &\leq \int_G g(z)d\|S\|(z) + \int_G h(z)d\|R\|(z)\\
  &\leq  (m+1)\Lipconst(\gamma) \int_0^1 \int_G g(z\gamma(t)) \bar{\tau}(|\gamma(t)|)^m d\|T\|(z)dt\\
 &\quad +  m\Lipconst(\gamma) \int_0^1 \int_G h(z\gamma(t)) \bar{\tau}(|\gamma(t)|)^{m-1} d\|\partial T\|(z)dt\\
  &\leq \int_0^1\bar{v}\circ\gamma(t) |\dot{\gamma}(t)|dt.
 \end{split}
\end{equation*}
Now suppose that $u(a) = u(b)=\infty$. If there exists a point $c$ in the image of $\gamma$ such that $u(c)<\infty$ then it follows as above (by replacing $a$ by $c$) that $$\int_0^1\bar{v}\circ\gamma(t) |\dot{\gamma}(t)|dt=\infty,$$ and this clearly also holds if $u=\infty$ everywhere on the image of $\gamma$. This shows that $\bar{v}$ is an upper gradient for $u$.
 Since every ball in $G$ of sufficiently small radius (independent of the center) admits a weak $1$-Poincar\'e inequality it follows from \cite[Theorem 1.11]{Jarvenpaaetal} that $u$ is measurable and locally integrable. Furthermore, by \cite{Shan}, a locally integrable function with locally $\kappa$-integrable upper gradient has a representative in $W^{1,\kappa}_{\operatorname{loc}}$. The proof is complete.
\end{proof}

Given $\omega$ and $T$ as in Lemma~\ref{lemma:sobolev-omega-translates} we may define
\begin{equation*}
 \omega_+(T, r):=  \frac{1}{\haus^n(B(e,r))}\int_{B(e,r)} |\omega(\varphi_{x\#}T)|\,d\haus^n(x)
\end{equation*}
for $r>0$. If furthermore $\omega$ is a linear cochain then we may define
\begin{equation*}
 \omega(T, r):=  \frac{1}{\haus^n(B(e,r))}\int_{B(e,r)} \omega(\varphi_{x\#}T)\,d\haus^n(x)
\end{equation*}
for $r>0$. We can estimate $\omega_+(T,r)$ and $|\omega(T,r)|$ as follows.

\begin{proposition}\label{proposition:avg-omega-estimate-curr}
 Let $G$ be a Lie group of dimension $n$,  endowed with a left-invariant Finsler metric. Let $0\leq m\leq n-1$ and $\mathcal{C} = (\mathcal{C}_m, \mathcal{C}_{m+1})$ with either $\mathcal{C}_k = \AKnc_k(G)$ or $\mathcal{C}_k=\AKic_k(G)$ for $k=m, m+1$.  Then the following properties hold:
 \begin{enumerate}
  \item[(i)] if $\omega\in\mathcal{W}_{q,p}(\mathcal{C}_m)$ for some $1\leq p,q<\infty$ and if $T\in\mathcal{C}_m$ then
  \begin{equation}\label{eq:omega-average-gen-flatnorm-norcurr}
  \omega_+(T,r) \leq \bar{\tau}(r)^m \flatnormFAT(T, \mathcal{C}) \left[\bar{\tau}(r)\haus^n(B(e,r))^{-\frac{1}{p}}\|g\|_p + \haus^n(B(e,r))^{-\frac{1}{q}}\|h\|_q\right]
 \end{equation}
 for all $r>0$, every upper norm $h$ and upper gradient $g$ of $\omega$ with respect to $\mathcal{C}_{m+1}$;
 \item[(ii)] if $\omega\in\mathcal{W}_{p}(\mathcal{C}_m^0)$ for some $1\leq p<\infty$ and if $T\in\mathcal{C}_m^0$ then
 \begin{equation}\label{eq:omega-average-gen-fillvol-norcurr}
  \omega_+(T,r) \leq \bar{\tau}(r)^{m+1} \haus^n(B(e,r))^{-\frac{1}{p}} \fillvol(T, \mathcal{C}_{m+1}) \|g\|_p
 \end{equation}
  for all $r>0$ and whenever $g$ is an upper gradient of $\omega$ with respect to $\mathcal{C}_{m+1}$;
  \item[(iii)] if $\mathcal{C}_k = \AKic_k(G)$ for $k=m, m+1$ and $T\in\mathcal{C}_m^0$ and if either $m=0$ or $G$ is a Carnot group of step $c$ or $\fillvol(T)\leq 1$ then $\|g\|_p$ in \eqref{eq:omega-average-gen-fillvol-norcurr} may be replaced by $\|g\|_{p, N(\spt T, r+\varrho(\fillvol(T)))}$ where
 \begin{equation*}
  \varrho(t) = \left\{\begin{array}{ll} Dt^{\frac{1}{m+1}} & 0<t\leq 1 \;\text{ or }\; m=0\\
   Dt^{\frac{c^{m+1}}{1+c +\dots + c^m}} & 1< t \;\text{ and }\; m\geq 1,
  \end{array}\right.
  \end{equation*}
with a constant $D$ depending only on $G$ and on the Finsler metric.
 \end{enumerate}
\end{proposition}

If $\omega$ is an addition a linear cochain then $\omega_+(T,r)$ can be replaced by $|\omega(T,r)|$ in statements (i), (ii), and (iii).

\begin{proof}
 Let $\omega$ be as in (i) and let $h$ be an upper norm and $g$ an upper gradient of $\omega$ with respect to $\mathcal{C}_{m+1}$.
 Let $U\in\mathcal{C}_m$ and $V\in\mathcal{C}_{m+1}$ be such that $T = U + \partial V$. Clearly, we have
 \begin{equation*}
  \|\varphi_{x\#} U\| \leq  \bar{\tau}(|x|)^m\varphi_{x\#}\|U\|
 \end{equation*}
 and
  \begin{equation*}
  \|\varphi_{x\#} V\| \leq  \bar{\tau}(|x|)^{m+1}\varphi_{x\#}\|V\|
 \end{equation*}
 for all $x\in G$. Together with Proposition~\ref{proposition:elem-upper-G} this yields
 \begin{equation*}
  \begin{split}
   \omega_+(T,r) & \leq \frac{1}{\haus^n(B(e,r))}\left(\int_{B(e,r)}|\omega(\varphi_{x\#}U)|\,d\haus^n(x) + \int_{B(e,r)}|\omega(\partial \varphi_{x\#}V)|\,d\haus^n(x)\right)\\
   &\leq \bar{\tau}(r)^m\frac{1}{\haus^n(B(e,r))} \int_{B(e,r)}\int_G h(zx)d\|U\|(z)d\haus^n(x)\\
   &\;\;\; + \bar{\tau}(r)^{m+1}\frac{1}{\haus^n(B(e,r))} \int_{B(e,r)}\int_G g(zx)d\|V\|(z)d\haus^n(x)\\
   &\leq \bar{\tau}(r)^m \left[\haus^n(B(e,r))^{-\frac{1}{q}} \|h\|_{q,\Omega_U} \mass(U) + \bar{\tau}(r) \haus^n(B(e,r))^{-\frac{1}{p}}\|g\|_{p,\Omega_V} \mass(V)\right]
  \end{split}
 \end{equation*}
 for every $r>0$, where $\Omega_U = (\spt U)\cdot B(e,r)$ and $\Omega_V = (\spt V)\cdot B(e,r)$. Taking the infimum over all $U$ and $V$ this yields \eqref{eq:omega-average-gen-flatnorm-norcurr} and proves (i). If $T\in\mathcal{C}_m^0$ then the above calculation with $U=0$ yields \eqref{eq:omega-average-gen-fillvol-norcurr} and thus (ii).
 
 We now prove statement (iii). In view of the inequality above, it is clearly enough to show that for every $\varepsilon>0$ there exists a filling $V\in\AKic_{m+1}(G)$ of $T$ satisfying $\mass(V)\leq (1+\varepsilon)\fillvol(T)$ and $$\spt V\subset N(\spt T, \varrho(\fillvol(T))).$$
 If $m=0$ then the existence of such $V$ follows from Lemma~\ref{lem:weak-structure-1-currents}. If $m\geq 1$ and $G$ is a Carnot group then the existence of such a $V$ is given by \cite[Proposition 4.3 and Corollary 7.3]{Wenger-asymptotic-rank}. Finally, suppose $m\geq 1$ and $\fillvol(T)\leq 1$. 
  By \cite{Wenger-flatconv}, there exists $D_0>0$ depending only on $G$ and the Finsler metric such that $G$ admits a Euclidean isoperimetric inequality for all cycles in $\AKic_m(G)$ of mass at most $D_0$. Let $\varepsilon\in(0,1)$ and let $S\in\AKic_{m+1}(G)$ be such that $\partial S = T$ and $\mass(S)\leq (1+\varepsilon)\fillvol(T)$. Define a $1$-Lipschitz function $\lambda(x):= \dist(\spt T, x)$. Set $\delta:=2D_0^{-1}\fillvol(T)$. By \cite[Theorems 5.6 and 5.7]{Ambrosio-Kirchheim-currents}, there exists $t\in(0, \delta)$ such that 
$$
\langle S, \lambda, t\rangle = \partial(S\rstr\{\lambda\leq t\}) - T
$$ 
is an integral current and has 
$$
\mass(\langle S, \lambda, t\rangle) \leq D_0.
$$
Indeed, otherwise we would have 
$$
\mass(S)\geq \|S\|(\{\lambda \leq \delta\})\geq \int_0^\delta \mass(\langle S, \lambda, t\rangle) dt > D_0\delta \geq 2\fillvol(T),
$$
 a contradiction. Set $T':=\langle S, \lambda, t\rangle$. Since $-\partial(S\rstr\{\lambda> t\}) = T'$ we clearly have $\fillvol(T')\leq \|S\|(\{\lambda>t\}).$
By \cite[Lemma 5.3 and its proof]{Wenger-flatconv} there exists a filling $S'\in\AKic_{m+1}(G)$ of $T'$ such that $$\mass(S')\leq (1+\varepsilon)\fillvol(T')$$
and
$$\spt S'\subset N(\spt T', D'\fillvol(T')^{\frac{1}{m+1}}),$$
where $D'$ only depends on $G$ and $d_0$.
It follows that $V:= S\rstr\{\lambda \leq t\} - S'$ is in $\AKic_{m+1}(G)$, has boundary $\partial V = T$, and satisfies $$\mass(V) \leq \|S\|(\{\lambda\leq t\}) + (1+\varepsilon)\fillvol(T') \leq (1+\varepsilon)\mass(S)\leq (1+\varepsilon)^2\fillvol(T)$$ and
\begin{equation}\label{eq:spt-V-small-fillvol}
\spt V\subset N(\spt T, D\fillvol(T)^{\frac{1}{m+1}}), 
\end{equation}
where $D = 4D_0^{-1} + 2D'$. Since $\varepsilon>0$ was arbitrary this completes the proof of the statement.
\end{proof}

Our next estimate is the following:

\begin{proposition}\label{proposition:tech-avg-diff-pt-omega-estimate}
 Let $G$ be a Lie group of dimension $n$, endowed with a left-invariant Finsler metric. Let $r_0>0$ and suppose $B(e,r_0)$ admits a $(D,s, \lambda, \eta)$-controlled family of curves. Let $0\leq m\leq n-1$ and $\mathcal{C} = (\mathcal{C}_m, \mathcal{C}_{m+1})$ with either $\mathcal{C}_k = \AKnc_k(G)$ or $\mathcal{C}_k=\AKic_k(G)$ for $k=m, m+1$. Let $T\in\mathcal{C}$ and suppose $T$ satisfies \eqref{eq:growth-mass-T} and \eqref{eq:growth-mass-bdry-T} for some $A,B>0$, $\alpha, \beta\in[0,s]$, $p> \max\{1, s-\alpha\}$, $q>\max\{1, s-\beta\}$, and all $r\in(0,2\lambda r_0)$. Then the following properties hold:
\begin{enumerate}
\item[(i)] if $\omega\in\mathcal{W}_{q,p}(\mathcal{C})$ then for every $r\in(0,r_0)$ and every upper norm $h$ and upper gradient $g$ of $\omega$ we have
\begin{equation*}
  \begin{split}
  |\,|\omega(T)| - \omega_+(T,r)\,| &\leq   D^2\eta\bar{\tau}(\lambda r)^{m-1}\left[ \bar{A} \bar{\tau}(\lambda r)r^{\frac{\alpha}{p}}\haus^n(B(e,r))^{-\frac{1}{p}}\mass(T)^{\frac{p-1}{p}} \|g\|_{p, \Omega}\right.\\
  &\quad\quad\quad\quad\quad\quad\quad\quad\quad + \left.\bar{B}r^{\frac{\beta}{q}} \haus^n(B(e,r))^{-\frac{1}{q}}\mass(\partial T)^{\frac{q-1}{q}} \|h\|_{q, \Omega}\right];
  \end{split}
 \end{equation*}
  \item[(ii)] if $\omega\in\mathcal{W}_{p}(\mathcal{C}^0)$ and $\fillvol(T, \mathcal{C}_{m+1})<\infty$ then for every $r\in(0,r_0)$ and every upper gradient $g$ of $\omega$ we have
  \begin{equation*}
   |\,|\omega(T)| - \omega_+(T,r)\,| \leq   D^2\eta\bar{\tau}(\lambda r)^m\left[ \bar{A} r^{\frac{\alpha}{p}}\haus^n(B(e,r))^{-\frac{1}{p}}\mass(T)^{\frac{p-1}{p}} \|g\|_{p, \Omega}\right].
 \end{equation*}
  \end{enumerate}
  In the inequalities, we have used
  \begin{equation*}
   \bar{A}:= \frac{(2\lambda)^{\frac{\alpha}{p}} A^{\frac{1}{p}}p(m+1)}{p-s+\alpha} \quad\text{ and }\quad \bar{B}:= \frac{(2\lambda)^{\frac{\beta}{q}} B^{\frac{1}{q}}qm}{q-s+\beta}
  \end{equation*}
  and $\Omega = (\spt T)\cdot B(e,\lambda r)$.
\end{proposition}

If $\omega$ is in addition linear then the inequalities in the proposition hold with $|\,|\omega(T)| - \omega_+(T,r)|$ replaced by $|\omega(T) - \omega(T,r)|$.

\begin{proof}
 We only prove (i) since the proof of (ii) is analogous. Let $H:[0,1]\times B(e, r_0)\to G$ be the Lipschitz map defining the controlled family of curves. Let $r\in(0,r_0)$. For $x\in B(e,r)$ define $\psi_x: [0,1]\times G\to G$ by $\psi_x(t,z):= zH_t(x)$ and note that $\psi_x(t,\cdot)$ is $\bar{\tau}(\lambda r)$-Lipschitz for every $t\in[0,1]$ and $\psi_x(\cdot, z)$ is $\eta$-Lipschitz for every $z\in G$. Define $S_x:= \psi_{x\#}([0,1]\times T)$. If $m\geq 1$ define $R_x:= \psi_{x\#}([0,1]\times\partial T)$; if $m=0$ then set $R_x:=0$. Note that $S_x\in\AKnc_{m+1}(G)$ and $R_x\in\AKnc_m(G)$ if $T\in\AKnc_m(G)$ and $S_x\in\AKic_{m+1}(G)$ and $R_x\in\AKic_m(G)$ if $T\in\AKic_m(G)$. Since  $\partial S_x = \varphi_{x\#}T - T - R_x$ and since, by Lemma~\ref{lemma:sobolev-omega-translates}, we have $|\omega(\varphi_{x\#}T)|<\infty$ for almost every $x\in B(e,r)$, we obtain
 \begin{equation*}
   |\,|\omega(T)| - |\omega(\varphi_{x\#}T)|\,| \leq |\omega(\partial S_x)| + |\omega(R_x)| \leq \int_G g(z)d\|S_x\|(z) + \int_Gh(z)d\|R_x\|(z)
 \end{equation*}
 for almost every $x\in B(e,r)$ and hence
 \begin{equation*}
  \begin{split}
   |\,|\omega(T)| - \omega_+(T,r)\,| &\leq \frac{1}{\haus^n(B(e,r))}\int_{B(e,r)} |\omega(T) - |\omega(\varphi_{x\#}T)||\,d\haus^n(x)\\
    &\leq \frac{1}{\haus^n(B(e,r))} \int_{B(e,r)} \int_G g(z)d\|S_x\|(z)\,d\haus^n(x) \\
    & \; + \frac{1}{\haus^n(B(e,r))}  \int_{B(e,r)} \int_G h(z)d\|R_x\|(z)\,d\haus^n(x).
  \end{split}
 \end{equation*}
 By Lemma~\ref{lemma:mass-estimate-push-cone}, we obtain
 \begin{equation*}
  \|S_x\| \leq (m+1) \eta \bar{\tau}(\lambda r)^m\,\psi_{x\#}(\lm^1\times \|T\|)
 \end{equation*}
 as well as 
 \begin{equation*}
   \|R_x\| \leq  m\eta \bar{\tau}(\lambda r)^{m-1}\,\psi_{x\#}(\lm^1\times \|\partial T\|)
 \end{equation*}
 for all $x\in B(e,r)$. Proposition~\ref{proposition:growth-integration-variant} yields 
 \begin{equation*}
  \begin{split}
  \int_{B(e,r)} \int_G g(z)d\|S_x\|&(z)\,d\haus^n(x) \\
   &\leq  (m+1)\eta \bar{\tau}(\lambda r)^m  \int_{B(e,r)} \int_0^1\int_G g(zH_t(x))d\|T\|(z)dt\,d\haus^n(x)\\
   &\leq  (m+1)D^2\eta \bar{\tau}(\lambda r)^m\frac{(2\lambda)^{\frac{\alpha}{p}} A^{\frac{1}{p}}p}{p-s+\alpha}r^{\frac{\alpha}{p}} \left[\haus^n(B(e,r))\mass(T)\right]^{\frac{p-1}{p}} \|g\|_{p, \Omega},
  \end{split}
 \end{equation*}
 where $\Omega = (\spt T) \cdot B(e,\lambda r)$. Similarly, we obtain
  \begin{equation*}
  \begin{split}
  \int_{B(e,r)} \int_G h(z)d\|R_x\|&(z)\,d\haus^n(x) \\
  &\leq  m\eta \bar{\tau}(\lambda r)^{m-1}  \int_{B(e,r)} \int_0^1\int_G h(zH_t(x))d\|\partial T\|(z)dt\,d\haus^n(x)\\
   &\leq mD^2\eta \bar{\tau}(\lambda r)^{m-1}\frac{(2\lambda)^{\frac{\beta}{q}} B^{\frac{1}{q}}q}{q-s+\beta}r^{\frac{\beta}{q}} \left[\haus^n(B(e,r))\mass(\partial T)\right]^{\frac{q-1}{q}} \|h\|_{q, \Omega}.
  \end{split}
 \end{equation*}
 Combining the above estimates gives the claim.
\end{proof}

We are finally ready to prove Theorems~\ref{hoeldercont2} and \ref{hoeldercont1}. We first give the proof of the latter theorem.

\begin{proof}[Proof of Theorem~\ref{hoeldercont1}]
 There exist $D\geq 1$ and $r_0>0$ such that $B(e,r)$ admits a $(D, n, D, Dr)$-controlled family of curves for every $0<r\leq r_0$ and such that
 $$\haus^n(B(e,r))\geq D^{-1} r^n$$ for all $0\leq r\leq r_0$. We may of course assume that $r_0\leq 1$. Note that there exists $D'$ such that $\bar{\tau}(Dr)\leq D'$ for all $0\leq r\leq r_0$. 
It now follows from Propositions~\ref{proposition:avg-omega-estimate-curr} and \ref{proposition:tech-avg-diff-pt-omega-estimate} that for every upper norm $h$ of $\omega$ and every upper gradient $g$ of $\omega$ with respect to $\mathcal{C}_{m+1}$ we have
\begin{equation*}
 \begin{split}
 |\omega(T)| &\leq |\,|\omega(T)|-\omega_+(T,r)\,|+ \omega_+(T,r) \\
 &\leq E' \left[\frac{A^{1/p}p}{p-n+\alpha} \mass(T)^{(p-1)/p}r^{1+\frac{\alpha}{p}} + \flatnormFAT(T, \mathcal{C})\right] r^{-\frac{n}{p}}\norm{g}_{p}\\
&\quad+ E'\left[\frac{B^{1/q}q}{q-n+\beta}\mass(\partial T)^{(q-1)/q} r^{1+\frac{\beta}{q}}+ \flatnormFAT(T, \mathcal{C})\right] r^{-\frac{n}{q}}\norm{h}_{q}\\
&\leq E' \left[\frac{A^{1/p}}{1-\gamma} (1 + \nmass(T))^{(p-1)/p}r^{1+\frac{\alpha-n}{p}} + \flatnormFAT(T, \mathcal{C})r^{-\frac{n}{p}}\right] \norm{g}_{p}\\
&\quad+ E'\left[\frac{B^{1/q}}{1-\gamma}(1+\nmass(T))^{(q-1)/q} r^{1+\frac{\beta-n}{q}}+ \flatnormFAT(T, \mathcal{C})r^{-\frac{n}{q}}\right] \norm{h}_{q}\\
\end{split}
\end{equation*}
for $0<r\leq r_0$, where $E'$ is a constant only depending on $D$, $D'$, and $m$, and where $\nmass(T) = \mass(T) + \mass(\bdry T)$. 
Set $$r:= \left(\frac{\flatnormFAT(T,\mathcal{C})}{1+\nmass(T)}\right)^{\frac{1}{1+\delta}}r_0$$ and note that $r\leq r_0$ since $\flatnormFAT(T, \mathcal{C})\leq \mass(T)$. With this choice of $r$ the inequality above easily yields
\begin{equation*}
 \begin{split}
  |\omega(T)|&\leq E\left[\frac{A^{1/p}}{1-\gamma} (1 + \nmass(T))^{1-\frac{1}{p}- \frac{1+\frac{\alpha-n}{p}}{1+\delta}}\flatnormFAT(T,\mathcal{C})^{\frac{1+\frac{\alpha-n}{p}}{1+\delta}}\right] \norm{g}_p\\
  &\quad +E\left[ (1 + \nmass(T))^{\frac{n}{p(1+\delta)}}\flatnormFAT(T, \mathcal{C})^{1- \frac{n}{p(1+\delta)}}\right] \norm{g}_{p}\\
&\quad + E \left[\frac{B^{1/q}}{1-\gamma}(1+\nmass(T))^{1-\frac{1}{q}- \frac{1+\frac{\beta-n}{q}}{1+\delta}}\flatnormFAT(T,\mathcal{C})^{\frac{1+\frac{\beta-n}{q}}{1+\delta}}\right]\norm{h}_q\\
&\quad + E \left[(1+\nmass(T))^{\frac{n}{q(1+\delta)}}\ \flatnormFAT(T, \mathcal{C})^{1-\frac{n}{q(1+\delta)}}\right]\norm{h}_{q}\\
 \end{split}
\end{equation*}
for some constant $E$ depending only on $E'$ and $r_0$. Since the exponents of $\flatnormFAT(T,\mathcal{C})$ are all between $\frac{1-\gamma}{1+\delta}$ and $\frac{1+\delta-\lambda}{1+\delta}$ and the exponents of $(1+\nmass(T))$ are bounded above by $\frac{\gamma+\delta}{1+\delta}$ we obtain from the above inequality that
\begin{equation*}
 |\omega(T)|\leq E \Lambda(T)\, \Big(\flatnormFAT(T, \mathcal{C})^{1-\frac{\lambda}{1+\delta}} + \flatnormFAT(T, \mathcal{C})^{1 - \frac{\gamma+\delta}{1+\delta}}  \Big)\,(\norm{g}_p + \norm{h}_q).
\end{equation*}
Since $h$ and $g$ were arbitrary the proof is complete.
\end{proof}

The proof of Theorem~\ref{hoeldercont2} is similar to the proof above but moreover uses the Euclidean isoperimetric inequality.

\begin{proof}[Proof of Theorem~\ref{hoeldercont2}]
There exist $D\geq 1$ and $r_0>0$ such that $B(e,r)$ admits a $(D, n, D, Dr)$-controlled family of curves for every $0<r\leq r_0$ and such that
 $$\haus^n(B(e,r))\geq D^{-1} r^n$$ for all $0\leq r\leq r_0$. We may of course assume that $r_0\leq 1$. Note that there exists $D'$ such that $\bar{\tau}(Dr)\leq D'$ for all $0\leq r\leq r_0$. Note furthermore that if $G$ is a normed space then the above holds with $r_0=\infty$.
Now, Propositions~\ref{proposition:avg-omega-estimate-curr} and \ref{proposition:tech-avg-diff-pt-omega-estimate} yield
\begin{equation}\label{eq:omega-T-est-initial-theorem-proof-hoelder2}
 \begin{split}
 |\omega(T)| &\leq |\,|\omega(T)|-\omega_+(T,r)\,|+ \omega_+(T,r) \\
 &\leq F r^{-\frac{n}{p}} \left[\frac{A^{\frac{1}{p}}p}{p-n+\alpha} \mass(T)^{\frac{p-1}{p}} r^{1+\frac{\alpha}{p}}  + \fillvol(T)\right]\|g\|_{p, N(\spt T, F \fillvol(T)^{\frac{1}{m+1}} + Dr)}
\end{split}
\end{equation}
for $0<r\leq r_0$, where $F$ is a constant only depending on $G$ and the left-invariant Finsler metric $d_0$ on $G$. Of course, we may assume that $T\not=0$. Now, suppose first that $G$ is a normed space. Setting $$r:= \left[\fillvol(T)^pA^{-1}\mass(T)^{1-p}\right]^{\frac{1}{p+\alpha}},$$ the above inequality becomes
\begin{equation}\label{eq:omega-T-fillvol-estimate-proof-theorem}
|\omega(T)|\leq F\left(1 + \frac{p}{p+\alpha - n}\right)\fillvol(T)^{1-n/(p+\alpha)}A^{\eta}\mass(T)^{(p-1)\eta} \norm{g}_{p, N(\spt T, t)}, 
\end{equation}
where $\eta = n/(p(p+\alpha))$ and $$t = F \cdot \fillvol(T)^{\frac{1}{m+1}} + Dr.$$ If $m=0$ then we clearly have $t\leq s_0$, where $s_0$ is as in the statement of the theorem. If $m\geq 1$ then it follows from the Euclidean isoperimetric inequality that $t\leq s_0$. This proves the theorem for the case that $G$ is a normed space.

Next, suppose that $G$ is arbitrary and $\fillvol(T)\leq 1$. If $m=0$ then set $D'':= r_0$. If $m\geq 1$ then define $D''$ as follows. By \cite{Wenger-flatconv}, there exists $0<D_0\leq 1$ such that $G$ admits a Euclidean isoperimetric inequality for cycles in $\AKic_m(G)$ of mass at most $D_0$. Denote by $\bar{D}$ the isoperimetric constant. We may assume that $\bar{D}\geq 1$. Set $D'':= \min\{\bar{D}^{-1}, D_0\} r_0$. Finally, define
\begin{equation}\label{eq:def-r-proof-hoeldercont2}
 r:= D''\left[\fillvol(T)^pA^{-1}\mass(T)^{1-p}\right]^{\frac{1}{p+\alpha}}.
\end{equation}
We claim that $r\leq r_0$. Indeed, if $m=0$ then $\mass(T)\geq 2$ and thus we clearly have $r\leq r_0$. If $m\geq 1$ and $\mass(T)>D_0$ then $$\left[\fillvol(T)^pA^{-1}\mass(T)^{1-p}\right]^{\frac{1}{p+\alpha}} \leq \left(\frac{1}{D_0}\right)^{\frac{p-1}{p+\alpha}} \leq D_0^{-1}$$ and hence $r\leq r_0$, as claimed. If $m\geq 1$ and $\mass(T)\leq D_0$ then, by the Euclidean isoperimetric inequality,  $$\left[\fillvol(T)^pA^{-1}\mass(T)^{1-p}\right]^{\frac{1}{p+\alpha}} \leq \left[\bar{D}^p A^{-1}\mass(T)^{1+\frac{p}{m}}\right]^{\frac{1}{p+\alpha}}\leq \bar{D}^{\frac{p}{p+\alpha}} \leq \bar{D}$$ and thus $r\leq r_0$, as claimed. With $r$ as in \eqref{eq:def-r-proof-hoeldercont2} it is not difficult to see that \eqref{eq:omega-T-est-initial-theorem-proof-hoelder2} becomes
\begin{equation*}
|\omega(T)|\leq E\left(1 + \frac{p}{p+\alpha - n}\right)\fillvol(T)^{1-n/(p+\alpha)}A^{\eta}\mass(T)^{(p-1)\eta} \norm{g}_{p, N(\spt T, t)}, 
\end{equation*}
 with a constant $E$ depending only on $G$ and $d_0$, and where again $$t = F \cdot \fillvol(T)^{\frac{1}{m+1}} + Dr.$$ If $m=0$ then clearly $t\leq s_0$. If $m\geq 1$ and $\mass(T)>D_0$ then a straightforward calculation shows that $t\leq s_0$. Finally, if $m\geq 1$ and $\mass(T)\leq D_0$ then the Euclidean isoperimetric inequality for cycles in $\AKic_m(G)$ of mass at most $D_0$ also gives that $t\leq s_0$. This completes the proof of the theorem.
 \end{proof}


\subsection{Families with positive modulus or capacity} \label{section:lie-haus-cap}
In Section \ref{section:rel-haus-cap}, we showed that the capacity of a set of currents vanishes if the currents are supported in a small enough set. The assumptions on the underlying metric space were mild. On the other hand, lower bounds or even positivity of capacities do not hold in general unless the underlying metric space has some structure. In this section we show that in the case of Lie groups, a single current with suitable local mass growth has non-zero $p$-capacity for large enough $p$. This property is closely connected to continuity, and has already implicitly appeared in the proofs of our main results above. We also give an example illustrating the sharp exponent $p$ for which this property holds.

We begin with the following elementary observation.

\begin{proposition}
\label{elementary}
 Let $G$ be a Lie group  of dimension $n$, endowed with a left-invariant Finsler metric, and let $T\in\AKc_m(G)$ with $T\not=0$ and $0\leq m\leq n$. Let $B\subset G$ be a Borel set with $\haus^n(B)>0$. Then the set $\Gamma:=\{ \varphi_{x\#}T: x\in B\}$ has
$M_q(\Gamma)>0$  for every $q \geq 1$.
\end{proposition}

\begin{proof}
 We may assume without loss of generality that $B$ is contained in a ball $B(e,R)$.
 We argue by contradiction and suppose that $M_q(\Gamma)=0$  for some $q \geq 1$. There then exists $f\in L^q(G)$ with $f\geq 0$ and such that $$\int_G f(z)d\|\varphi_{x\#}T\|(z)=\infty$$ for every $x\in B$. Since $$\|\varphi_{x\#}T\|\leq \bar{\tau}(|x|)^m\varphi_{x\#}\|T\| \leq \bar{\tau}(R)^m\varphi_{x\#}\|T\|$$ for every $x\in B$, H\"older's inequality and Lemma~\ref{lemma:fubini-G} imply
 \begin{equation*}
  \begin{split}
  \infty &= \int_B\int_G f^q(z)d\|\varphi_{x\#}T\|(z)d\haus^n(x)\\
  & \leq \bar{\tau}(R)^m \int_G f^q(x)\|T\|(xB^{-1})d\haus^n(x)\\
  & \leq  \bar{\tau}(R)^m \mass(T) \|f\|_q^q,
  \end{split}
 \end{equation*}
 contradicting the fact that $f\in L^q(G)$.
 \end{proof}

\begin{proposition}
\label{caplowerbound}
 Let $G$ be a Lie group of dimension $n$, endowed with a left-invariant Finsler metric. Let $0\leq m\leq n-1$ and $T\in\AKic_m^0(G)$ with $T\not=0$ and $\fillvol(T)<\infty$. Suppose there exist $A, r_1>0$ and $\alpha\in[0,n]$ such that $T$ satisfies \eqref{eq:growth-mass-T} for every $r\in(0,r_1)$. Then we have 
 $$\capa_p(\{T\}, \AKic_{m+1}(G))>0$$
 for every $p>n-\alpha$. If $T\in\AKnc_m^0(G)$ and if $T$ satisfies the same conditions as above then 
 $$\capa_p(\{T\}, \AKnc_{m+1}(G))>0$$
 for every $p>n-\alpha$. 
\end{proposition}

\begin{proof}
 Clearly, there exist $C\geq 1$ and $r_0>0$ such $B(e,r_0)$ admits a $(C,n,C, Cr_0)$-controlled family of curves. We may assume that $Cr_0<r_1$. Let $H:[0,1]\times B(e, r_0)\to G$ be the Lipschitz map defining the controlled family of curves. For $x\in B(e,r_0)$ define a Lipschitz map $\psi_x: [0,1]\times G\to G$ by $\psi_x(t,z):= zH_t(x)$. 
 Let $f:G\to[0,\infty]$ be a Borel measurable function such that
 \begin{equation*}
  \int_Gf(z)d\|S\|(z) \geq 1
 \end{equation*}
 for every $S$ with $\partial S = T$. Fix $V$ with $\partial V = T$. Such $V$ exists by assumption. For $x\in B(e, r_0)$ define 
 $$S_x:= \varphi_{x\#} V - \psi_{x\#}([0,1]\times T)$$ and note that $\partial S_x = T$ and, by Lemma~\ref{lemma:mass-estimate-push-cone},
 \begin{equation*}
   \|S_x\| \leq \|\varphi_{x\#}V\| + \|\psi_{x\#}([0,1]\times T)\|
    \leq D \varphi_{x\#}\|V\| + D\psi_{x\#}(\lm^1\times \|T\|),
  \end{equation*}
 where $D= \max\left\{\bar{\tau}(Cr_0)^{m+1}, (m+1)Cr_0\bar{\tau}(Cr_0)^m\right\}$. Hence, by Propositions~\ref{proposition:elem-upper-G} and \ref{proposition:growth-integration-variant}, we obtain
  \begin{equation*}
   \begin{split}
   \haus^n(B(e,r_0)) &\leq \int_{B(e,r_0)}\int_Gf(z) d\|S_x\|(z)d\haus^n(x)\\
    &\leq D \int_{B(e,r_0)}\int_Gf(zx) d\|V\|(z)d\haus^n(x) \\
    &\quad + D \int_{B(e,r_0)}\int_0^1\int_Gf(zH_t(x)) d\|T\|(z)dt\,d\haus^n(x)\\
    &\leq D \|f\|_p \haus^n(B(e,r_0))^{\frac{p-1}{p}}\left[\mass(V) + \frac{2^{\frac{\alpha}{p}}C^{2+\frac{\alpha}{p}}A^{\frac{1}{p}}p}{p-n+\alpha}r_0^{\frac{\alpha}{p}}\mass(T)^{\frac{p-1}{p}}\right].
   \end{split}
  \end{equation*}
 This shows that $\|f\|_p$ is bounded away from $0$ and since $f$ was arbitrary we find that the capacity is also bounded away from $0$. This concludes the proof.
\end{proof}



\begin{proposition}
\label{examplezerocap}
Given $n \geq 2$ and $0 \leq \alpha \leq m \leq n-1$, and $1 \leq p< n- \alpha$, there exist $A>0$  and a non-zero current $T \in \AKic_m^0(\R^n)$ such that 
\begin{equation}\label{eq:upper-bound-growth-examplezerocap}
\|T\|(B(x,r))\leq Ar^\alpha
\end{equation}
for every $x\in\R^n$ and every $r\geq 0$ and such that
\begin{equation}
\label{kilpel}
\capa_p(\{T\}, \AKic_{m+1}(\R^n))=0.
\end{equation}
\end{proposition}

Note that if $p>1$ then $T$ in Proposition~\ref{examplezerocap}, in particular, satisfies \eqref{eq:growth-mass-T} for every $r\geq 0$.

\begin{proof}
If $m=0$ then it suffices to choose $T=\Lbrack x_1\Rbrack - \Lbrack x_0\Rbrack$ for arbitrary $x_0,x_1\in\R^n$ with $x_1\not=x_0$. Indeed, for such $T$ it follows from Theorem~\ref{thm:hausdim-capacity-currents} and the remark after the theorem that $\capa_p(\{T\}, \AKic_1(\R^n))=0$.

If $m \geq 1$, we may assume that $\alpha>0$. Fix $m$, $n$, and $\alpha$ and $p$ as above. Moreover, let $r_j=2^{-j}$, $j \in \mathbb{N}$. Let $\varphi: \R^{m+1}\to \R^n$ be defined by $\varphi(x_1,\dots, x_{m+1}):= (x_1,\dots, x_{m+1}, 0,\dots, 0)$. Denote by $B^{m+1}$ the unit ball in $\R^{m+1}$ and set $T_0=\partial \varphi_{\#}\Lbrack 1_{B^{m+1}}\Rbrack$. Let $M_j$ and $N_j$ be integers whose precise values will be determined later. Finally, choose points $x_j^k=(j,k,0,\ldots,0)\in\R^n$ for $k =1, \ldots, M_j$. Note that the balls $B(x_j^k,2r_j)$ are then pairwise disjoint for every $j$ and $k$. We define 
$$
T=\sum_{j=1}^{\infty} \sum_{k=1}^{M_j} N_jF_{\#}^{j,k}T_0,  
$$ 
where $F^{j,k}(x)=x_j^k+2^{-j} x$. We now choose $M_j$ and $N_j$ so that $T$ has finite mass (which implies that $T \in \AKic_m^0(\R^n)$) and such that $T$ satisfies the growth condition $\|T\|(B(x,r))\leq Ar^{\alpha}$ for a suitable constant $A$. 
For this, we first choose $N_j$ to be the largest integer smaller than or equal to $r_j^{\alpha -m}$. By disjointness of the balls $B(x_j^k,2r_j)$, we have
$$
\|T\|(B(x_j^k,r_j)) \leq m\omega_m N_j r_j^m \leq m \omega_m r_j^{\alpha}
$$
and it thus follows that 
$$
\|T\|(B(x,r)) \leq C r^{\alpha} 
$$
for every $r>0$ and $x \in \R^n$. 
Next, let $M_j$ be the largest integer smaller than or equal to $j^{-2}r_j^{-\alpha}$. 
Then 
$$
\mass(T) \leq m\omega_m \sum_{j=1}^{\infty} M_j r_j^{\alpha} \leq m\omega_m \sum_{j=1}^{\infty} j^{-2} < \infty. 
$$

We now show that \eqref{kilpel} holds. Notice that, if $R \in \AKic_{m+1}(\R^n)$ is such that 
$\partial R = F^{j,k}_{\#}T_0$ for some $j$ and $k$, then 
$$
\|R\|(B(x_j^k,2r_j)) \geq C r_j^{m+1}. 
$$
If follows that, if $S \in \AKic_{m+1}(\R^n)$ is such that $\partial S = T$, and if $g$ is defined as 
$$
g(x)=\sum_{j=1}^{\infty} M_j^{-1}N_j^{-1} r_j^{-m-1} \sum_{k=1}^{M_j} 1_{B(x_j^k,2r_j)}(x), 
$$ 
then 
\begin{equation}
\label{testfcn}
\int_{\R^n} g \, d \|S\| = \infty. 
\end{equation}
On the other hand,
\begin{eqnarray*}
\int_{\R^n} g(x)^p \, dx \leq C \sum_{j=1}^{\infty} M_j^{1-p} N_j^{-p} r_j^{n-p(m+1)} 
= C \sum_{j=1}^{\infty} j^{-2(1-p)} r_j^{n-p-\alpha}. 
\end{eqnarray*}
Since $r_j=2^{-j}$ and $p < n-\alpha$, the series converges. So $g$ is $p$-integrable. Since $\epsilon g$ is a test function for the capacity for every $\epsilon >0$ by \eqref{testfcn}, we conclude that \eqref{kilpel} holds. 
\end{proof}

\begin{remark}
If $m=\alpha=n-1$, then Proposition \ref{caplowerbound} holds with $p=1$ by Proposition \ref{elementary}. On the other hand, if 
$m=\alpha \leq n-2$, then the proposition does not hold with $p=n-m$, see Example \ref{sharpthm}. For other values of $\alpha$, we do not know if Proposition \ref{caplowerbound} holds with the borderline exponent $p=n-\alpha$. 
\end{remark}


\vskip 15pt

\noindent K.R.\quad University of Jyv\"askyl\"a, Department of Mathematics and Statistics (P.O. Box 35), FI-40014 University of Jyv\"askyl\"a, Finland\\
e-mail: kai.i.rajala@jyu.fi

\vskip 5pt

\noindent S.W.\quad Universit\'e de Fribourg, Math\'ematiques, Ch. du Mus\'ee 23, 1700 Fribourg, Switzerland\\
e-mail: stefan.wenger@unifr.ch

\end{document}